\documentclass[a4paper,reqno]{amsart}
\usepackage{hyperref}

\usepackage{amssymb,amsxtra}
\usepackage{graphicx} \usepackage{amsmath} \usepackage{amssymb}
\usepackage[svgnames]{xcolor}
\usepackage[color,matrix,arrow]{xy}
\input xy \xyoption{all}

\newtheorem{lemma}{Lemma}[section]
\newtheorem{theorem}[lemma]{Theorem}
\newtheorem{proposition}[lemma]{Proposition}
\newtheorem{corollary}[lemma]{Corollary}

\newcommand{\slrw}[1]{\stackrel{#1}{\longrightarrow}}
\newcommand{\lrw}{\longrightarrow}
\newcommand{\LL}{\Lambda}

\newcommand{\bfa}{\mathbf a}
\newcommand{\bfi}{\mathbf i}
\newcommand{\bfj}{\mathbf j}
\newcommand{\bfv}{\mathbf v}
\newcommand{\bfb}{\mathbf b}

\newcommand{\bfe}{\mathbf e}
\newcommand{\mR}{\mathcal R}
\newcommand{\Zp}{{\mathbb Z^+}}
\newcommand{\mRp}{\mathcal R^+}
\newcommand{\bbfe}{\bar{\mathbf e}}
\newcommand{\bbfg}{\bar{\mathbf g}}
\newcommand{\bfg}{\mathbf g}

\newcommand{\bfu}{\mathbf u}

\newcommand{\bvi}[1]{\bfv^{\bfi}(#1)}
\newcommand{\bbvi}[1]{\bar{\bfv}^{\bfi}(#1)}
\newcommand{\bvii}{\bfv^{\bfi}}

\newcommand{\gkk}[2]{\gamma^{(#1)}_{#2}}

\newcommand{\zk}[2]{\zeta^{(#1)}_{#2}}
\newcommand{\zks}[1]{z_{#1}}

\newcommand{\ta}[2]{\theta^{(#1)}_{#2}}

\newcommand{\ep}[2]{\epsilon^{(#1)}_{#2}}
\newcommand{\tL}{\tilde{\Lambda}}
\newcommand{\otau}{\overline{\tau}}
\newcommand{\olL}{\overline{\Lambda}}
\newcommand{\GG}{\Gamma}

\newcommand{\olG}{\overline{\Gamma}}
\newcommand{\tS}{\tilde{S}}
\newcommand{\tP}{\tilde{P}}
\newcommand{\tQ}{\tilde{Q}}

\newcommand{\tI }{\tilde{I}}
\newcommand{\olS}{\overline{S}}
\newcommand{\olP}{\overline{P}}
\newcommand{\olQ}{\overline{Q}}
\newcommand{\olI }{\overline{I}}
\newcommand{\olrho }{\overline{\rho}}

\newcommand{\sk}[1]{^{(#1)}}

\newcommand{\soc}{\mathrm{soc}\,}

\newcommand{\uhom}[1]{\underline{\mathrm{Hom}}}

\newcommand{\Ker}{\mathrm{Ker}\,}

\newcommand{\add}{\mathrm{add}\,}

\newcommand{\bfr}{{\mathbf r}}

\newcommand{\cP}{\mathcal{P}}
\newcommand{\cI}{\mathcal{I}}

\newcommand{\dmn}{\mathrm{dim}\,}
\newcommand{\arr}[2]{\begin{array}{#1}#2\end{array}}
\newcommand{\mat}[2]{\left(\begin{array}{#1}#2\end{array}\right)}
\newcommand{\cmb}[2]{\mat{c}{#1\\#2}}
\newcommand{\eqqc}[2]{\begin{equation}\label{#1}#2\end{equation}}

\begin{document}

\title{On $n$-cubic Pyramid  Algebras} 

\address{Jin Yun Guo, Deren Luo \\ Department of Mathematics and Key Laboratory of HPCSIP (Ministry of Education of China), Hunan Normal  University, Changsha, CHINA}
\email{gjy@hunnu.edu.cn}
\thanks{This work is partly supported by Natural Science Foundation of China \#11271119, and by  Provincial Innovation Foundation For Postgraduate of Hunan \#CX2014B189}



\maketitle

\begin{abstract}
In this paper we study a class of algebras having $n$-dimensional pyramid shaped quiver with $n$-cubic cells, which we called $n$-cubic pyramid algebras.
This class of algebras includes the quadratic dual of the basic $n$-Auslander absolutely $n$-complete algebras introduced by Iyama.
We show that the projective resolution of the simples of $n$-cubic pyramid algebras can be characterized by $n$-cuboids, and prove that they are periodic.
So these algebras are almost Koszul and $(n-1)$-translation algebras.
We also recover Iyama's cone construction for $n$-Auslander absolutely $n$-complete algebras using $n$-cubic pyramid algebras and the theory of $n$-translation algebras.
\end{abstract}

\section{Introduction}
\label{intro}
Quivers, especially translation quivers, are very important in the representation theory of algebras \cite{ass,ars,r1}.
There are many algebras, such as Auslander algebras, preprojective algebras,   related to translation quivers \cite{ars,dr,bgl}.
They are widely used in many mathematical fields, such as Cohen-Maucauley modules, cluster algebras, Calabi-Yau algebras and categories, non-commutative algebraic geometry and mathematical physics \cite{a0,birs,k1,rt1,gls,ceg,af}.
Recently, Iyama has developed the higher representation theory \cite{i1,i2,i4}, where a class of higher translation quivers also plays an important role.
In \cite{i4}, he characterizes a class of higher representation-finite algebras, $n$-Auslander absolutely $n$-complete algebras.
We observed that the quivers of such algebras are $n$-dimensional pyramid shaped quivers with $n$-cubic cells.
We study in this paper a class of algebras, call $n$-cubic pyramid algebras, defined on such quivers, which includes quadratic dual of $n$-Auslander absolutely $n$-complete algebras.
We draw the quiver in a different way to show that the arrows in such quiver in fact point to $n$-directions in a $n$-dimensional space.

Almost Koszul rings are introduced by Brenner, Butler and King in \cite{bbk}.
One of their main results in the paper is the periodicity of the trivial extensions of representation finite hereditary algebras of bipartite oriented quiver.
In fact, they prove that such algebras are almost Koszul of type $(3, h-1)$ for the Coxeter number $h$ of the quiver of the algebra.
Our aim is to generalize a modified version of this result to $n$-cubic pyramid  algebras.
We proved that a stable $n$-cubic pyramid algebra of height $m \ge 3$, that is, a twisted trivial extension of a corresponding $n$-cubic pyramid algebra, is almost Koszul of type $(n+1, m-1)$.
Our proof is based on the combinatoric characterization of the projective resolutions of the simple modules of the stable $n$-cubic pyramid algebras, using the integral points on an $(n+1)$-cuboid defined on the corresponding vertex in the quiver.
Observe that the class of trivial extensions of hereditary algebras of bipartite oriented quiver is the same as the trivial extensions of the quadratic dual of the hereditary algebras.
Our result can be regarded as a higher representation theory version of the result of Brenner, Butler and King \cite{bbk} in the case $A_m$.

In \cite{g02}, one of the authors introduces translation algebras as algebras with translation quivers as their quivers and the translation corresponds to an operation related to the  Nakayama functor.
Such algebras includes the quadratic dual of Auslander algebras, preprojective algebras.
In \cite{gu14}, $n$-translation quivers and $n$-translation algebras are introduced.
The classical construction of $\mathbb ZQ$ is generalized to construct an  $(n+1)$-translation algebra from an  $n$-translation algebra with admissible $n$-translation quiver, using trivial extension followed by a smash product with $\mathbb Z$.
In this paper, we construct $n$-cubic pyramid algebras inductively using  this method and a truncation called cuboid truncation, and study their relationship with higher representation theory introduced by Iyama \cite{i1,i4}.
By taken quadratic dual of a special case, we recover Iyama's cone construction of absolutely $(n+1)$-complete algebra $T\sk{n+1}_m(k)$ from an absolutely $n$-complete algebra $T\sk{n}_m(k)$ in \cite{i4}.

Our main theorem assert that an $n$-cubic pyramid algebra is an extendible $(n-1)$-translation algebra with admissible $(n-1)$-translation quiver.
In Section 2, we introduce $n$-cubic pyramid algebra $\Lambda(n)$, its stable version $\tL(n)$ and covering version $\olL(n)$, we also list their properties in term of $n$-translation algebras (Theorems 1 to 4), as consequences of our main theorem (Theorem ~\ref{main}, and of  an intermediate result, Lemma ~\ref{hproofof}). 
We prove our main theorem in the next three sections.
In Section 3, we prove that  $\tL(n)$ is almost Koszul algebra by explicitly constructing minimal projective resolution of its simples (Lemma ~\ref{hproofof}).
Under the assumption that $\Lambda(n) $ is an admissible $(n-1)$-translation algebra, this implies that it is extendible, and that $\tL(n)$ is a trivial extension of $\LL(n)$ and $\olL(n)$ is a smash product of $\tL(n)$, by \cite{gu14}.
Under this assumption, we show that $\LL(n+1)$ is a cuboid truncation of $\olL(n)$ and it is an $n$-translation algebra in Section 4.
Then we prove our main theorem by induction 
in Section 5.
In Section 6, we prove that both $\LL(n+1)$ and its quadratic dual have $n$-almost splitting sequence, and by taking quadratic dual of a special case, we recover Iyama's cone construction of absolutely $(n+1)$-complete algebra $T\sk{n+1}_m(k)$ from an absolutely $n$-complete algebra $T\sk{n}_m(k)$ in \cite{i4}.


\section{Pyramid Shaped $n$-cubic Quivers and Related  Algebras}
\label{sec:1}
Let $k$ be a field.
In this paper, algebra is assumed to be a graded quotient of a  path algebra of a locally finite quiver $Q$ over $k$, that is, $\Lambda = kQ/(\rho) = \Lambda_0 + \Lambda_1 + \cdots $, with $\Lambda_0$ a direct sum of (possibly, infinite) copies of $k$ and $\Lambda$ generated by $\Lambda_1$ over $\Lambda_0$, with relation set $\rho$ (see \cite{gu14}).
We assume the vertex set of $Q$ is $Q_0$ and the arrow set is $Q_1$.
So we have a complete set of idempotents $\{e_i\mid i\in Q_0\}$ such that $\Lambda_0 =\bigoplus\limits_{i\in Q_0} \Lambda_0 e_i$ and the number of arrows from $i$ to $j$ is exactly $\dmn_k e_j\LL_1 e_i$ for any $i,j \in Q_0$.

\medskip

To introduce pyramid shaped $n$-cubic quiver, we first make some convention on the non-negative integral vectors.

Given a positive integer $p$, let $\bfe_0=0$, $\bfe\sk{p}=(1, \ldots, 1)$ be the $p$-dimensional vector  with all the components $1$ and let $\bfe_t\sk{p} =(0, \ldots, 0, 1, 0, \ldots, 0)$ be the $p$-dimensional vector  with $t$th component $1$ and all the other components zero.
Write $\bfe\sk{p}_J=\bfe\sk{p}_{t_1,\ldots, t_s}=\sum\limits_{h=1}^{s} \bfe\sk{p}_{t_s} $ for a subset $J=\{t_1,\ldots, t_s\}$ of $ \{1, \ldots, p \}$, set $\bbfe\sk{p}_t=\bfe\sk{p}_{\{t, t+1\}}$.
Conventionally, we write them as $\bfe, \bfe_t, \bfe_J$ and $\bbfe_t$ when no confusion will occur.

For a vector denoted by the boldface letter such like $\bfa$ write $a_s$ for its $s$th component, so $\bfa = (a_1, \ldots, a_p)$.
For a vector $\bfi=(i_1, \ldots, i_p)\in \mathbb Z^p$, write $\bfi(1)= \bfi+\bfe_{1}$, $\bfi(t)= \bfi-\bfe_{t-1}+\bfe_{t}$ for $2\le t \le p$, and write $\bfi(t_1,\ldots,t_s)=\bfi(t_1)\ldots(t_s)$.
Similarly, write $(1)\bfi= \bfi-\bfe_{1}$, $(t)\bfi= \bfi+\bfe_{t-1}-\bfe_{t}$ for $2\le t \le p$, and write $(t_1,\ldots,t_s)\bfi=(t_s)\ldots(t_1)\bfi$.
Clearly $(t)\bfi(t)=\bfi$.

Let $\Zp $ be the set of non-negative integers.
For $\bfa =(a_1,\ldots,a_{p}) \in \Zp^p$, write $|\bfa| =\sum\limits_{t=1}^{p} a_t$.
Let $\Zp^{p}_l =\{\bfa =(a_1,\ldots,a_{p})\in \Zp^{p}\mid  |\bfa|=l\}$.
We need the following subsets of $\Zp^{p}$.
The vertex set
$$U\sk{p}=\{\bfu=(u_1,u_2,\ldots,u_{p})\mid 0\le u_t\le  1\}\subset \Zp^{p},$$ of the unit $p$-cube and the vertex sets $$U_l\sk{p}=\{\bfu=(u_1,u_2,\ldots,u_{p})\mid |\bfu|=l\},$$
of $l$-nets inside the unit $p$-cube, for $0 \le l \le p$.

\medskip

Fix $m \ge 3$ and  $\bfi=(i_1\ldots,i_p) \in \Zp^{p}$ satisfying \eqqc{vertexofQ}{1\le i_t \mbox{ for } s=1,\ldots,p \mbox{ and } |\bfi|\le m+p-1.
}
Set \eqqc{defineb}{b_1(\bfi)=m+p-1-|\bfi|, \quad  b_t(\bfi)=i_{t-1}-1, \quad 2\leq t\leq p+1,} write $b_t=b_t(\bfi)$ and $\bfb = \bfb(\bfi) =(b_1,\ldots,b_{p+1})$, then $|\bfb|=m-1$.
Note that when $1\le i_t $ for $ s=1,\ldots,p$, $|\bfi|\le m+p-1$ is equivalent to $\sum_{t=1}^s i_t \le m+s-1$ for $ s=1,\ldots,p.$
Let \eqqc{defineC}{C^{(p)}(\bfi)=\{(a_1,a_2,\ldots,a_{p+1})\mid 0\le a_t\le b_t(\bfi), 1\leq t\leq p+1 \}\subset \Zp^{p+1},} this is the set of integral vertices of  a $(p+1)$-cuboid with sides of length  $b_t$ for $1\leq t\leq p+1 $.
We call it the {\em $(p+1)$-cuboid associated to the vertex $\bfi$.}
The vertex $C^{(p)}_0(\bfi)=(0,\ldots,0)$ is called the initial vertex of the $C^{(p)}(\bfi)$.
Let $$C^{(p)}_l(\bfi)=\{\bfa =(a_1, \ldots, a_{p+1})\in C^{(p)}(\bfi) \mid \,|\bfa|=l\},$$ for $0\le l\le m-1$.
$C^{(p)}_l(\bfi)$ is a net inside the $(p+1)$-cuboid formed by integral vertices which can be reached from the initial vertex in $l$ steps.
We have that $C^{(p)}_{m-1}(\bfi) =\{\bfb =(b_1,\ldots,b_{p+1})\}$, and $C^{(p)}_l(\bfi) =\emptyset $ for $l\ge m$.

For $\bfi=(i_1,\ldots,i_p)\in\Zp^p$ satisfying \eqref{vertexofQ}, for each $\bfa =(a_1, \ldots, a_{p+1})\in \Zp^{p+1}$, define
$$\bvi{\bfa} =(i_1+a_1-a_2, i_2+a_2-a_3,\ldots, i_{p-1} + a_{p-1}-a_p, i_p+a_p-a_{p+1}),$$
For $\bfi=(i_1,\ldots,i_p, i_{p+1}) =(\bfi',i_{p+1})\in \Zp^{p+1}$ with $\bfi'=(i_1,\ldots,i_p)$ satisfying \eqref{vertexofQ}, write $C\sk{p}(\bfi) =C\sk{p}(\bfi')$.
For each $\bfa =(a_1, \ldots, a_{p+1})\in \Zp^{p+1}$, define
$$\bbvi{\bfa} =(i_1+a_1-a_2, i_2+a_2-a_3, \ldots, i_p+a_p-a_{p+1}, i_{p+1}+a_{p+1}).$$

\medskip

From now on, we fix  integers $n\ge 1$ and $m \ge 3$.
Define {\em $n$-cubic pyramid  quiver} $Q(n)$ of height $m $ as the quiver with \eqqc{thequiver}{\arr{ll}{\mbox{vertex set:}&
Q(n)_0=\{\bfi=(i_1,\ldots,i_n)\mid  1 \le i_t, |\bfi| \le m+n-1 \}
\\
\mbox{arrow set:}&
Q(n)_1=\{\gkk{t}{\bfi}:\bfi\lrw \bfi(t)\mid  1 \le t\le n, \bfi, \bfi(t) \in Q(n)_0\}.
}}
Call $\gkk{t}{\bfi}$ an {\em arrow of type $t$} starting at $\bfi$.
In the case of $n=3$, these quivers look like  a pyramid of side length of $m$ built up with cubes:

{\tiny
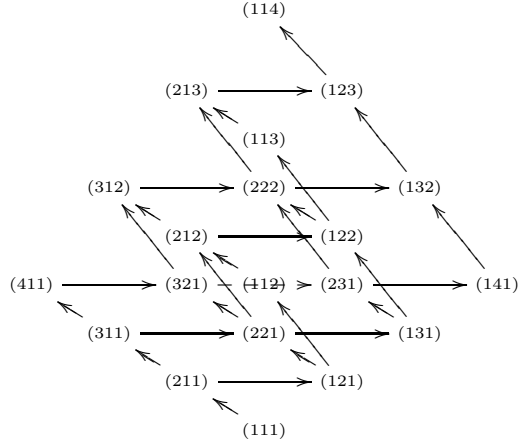
$$\xymatrix@C=0.2cm@R0.2cm{
&&&&&&&&&&&&&&&\\
&&&&&&&&&&&&&&&\\
& &&(114)\\
&&&&&&&&&&&&&&&\\
&& (213)\ar[rr]&&(123)\ar[uul]\\
&&&(113)\ar[ul]&&&&\\
&(312)\ar[rr]&& (222)\ar[uul]\ar[rr]&&(132)\ar[uul]\\
&&(212)\ar[rr]\ar[ul]&&(122)\ar[uul]\ar[ul]&&&&\\
(411)\ar[rr]&&(321)\ar[uul]\ar@{-->}[rr]&(112)\ar[ul]&(231)\ar[uul]\ar[rr]&&(141)\ar[uul]&& \\
&(311)\ar[ul]\ar[rr]&&(221)\ar[rr]\ar[ul]\ar[uul]&&(131)\ar[uul]\ar[ul]&&\\
&&(211)\ar[ul]\ar[rr]&&(121)\ar[ul]\ar[uul]&&&&&&&\\
&&&(111)\ar[ul]&&&&
 }$$}

These quivers are one of those defined in \cite{i4} inductively to describe the $n$-Auslander absolutely  $n$-complete algebras.
It is shown in \cite{g13} that such quivers are  truncation of the McKay quivers of some finite Abelian subgroups of $GL_n(\mathbb C)$.

For $\bfi=(i_1,\ldots,i_n) \in Q(n)_0$,  call the integral vector $\bfa \in \Zp^{n+1}$ an {\em $\bfi$-quiver vertex} if $\bvi{\bfa} \in Q(n)_0$.
Clearly, $\bvi{\bfa} $ is in $Q(n)_0$ if and only if $i_t > a_{t+1}-a_t$ and $a_{s+1}- a_1\le m+s-1-\sum\limits_{t=1}^s i_t$ for $1 \le s \le n$.

Let $\bfi$ be a vertex in $Q(n)$.
The {\em $n$-cubic cell at $\bfi$} in $Q(n)$ is the full subquiver $H^{\bfi}$ of $Q(n)$ with the vertex set the set of the $\bfi$-quiver vertices in $\{\bbvi{\bfa}\in Q(n)_0\mid \bfa \in U\sk{n}\}$.
The vertex $\bbvi{\bfe}$ is called {\em the end vertex} of $H^{\bfi}$ and $H^{\bfi}$ is called {\em complete} if $\bbvi{\bfe}$ is a vertex of $H^{\bfi}$.
$H^{\bfi}$ can be regarded as formed by certain vertices and the edges in the unit $n$-cube directed from $0$ to $\bfe$.
Define {\em $n$-hammock $H^{\bfi}_C$ at $\bfi$} as the full subquiver  of $Q(n)$ with the vertex set the set of the $\bfi$-quiver vertices in $\{\bbvi{\bfa}\in Q(n)_0\mid \bfa \in C\sk{n-1}(\bfi)\}$.
The vertex $\bbvi{\bfb(\bfi)}$ is called {\em the end vertex} of $H_C^{\bfi}$ and $H_C^{\bfi}$ is called {\em complete} if $\bbvi{\bfb(\bfi)}$ is a vertex of $H_C^{\bfi}$.

To define the relations in a general setting, we need some twists define by the following sequence of linear maps.
Let \eqqc{elementg}{\bfg =(g_1=1,g_2,\ldots,g_n)} be a sequence of linear maps on $kQ(n)$ such that the restriction on the vertex set of $g_s$ is  defined from  the vertices with $s$ component $i_s>1$ to the vertices with $\sum_{t=1}^{s} i_t <m+s-1$, in this case $g_s({\bfi})={\bfi-\bfe_s}$, or, $g_s(e_{\bfi})=e_{\bfi-\bfe_s}$, and $g_s(e_{\bfi})=0$ if $i_s=1$.
$g_s$ defines a linear map on $kQ_1$, the subspace of $kQ(n)$ spanned by the arrows of $Q(n)$ in such a way that $g_s$ restricts to a bijective linear map from  $e_{\bfi}kQ_1 e_{\bfj}$ to $e_{\bfi-\bfe_s}kQ_1 e_{\bfj-\bfe_s}$ if  $i_s\neq 1 \neq j_s$ and there is an arrow from $\bfj$ to $\bfi$, and it is $0$ on $e_{\bfi}kQ_1 e_{\bfj}$ otherwise.
Since there is at most one arrow between each pair of vertices, we have that $g_s(\gkk{t}{\bfi}) = d_{s,t,\bfi}\gkk{t}{g_s(\bfi)}$ for some $0\neq  d_{s,t,\bfi} \in k$ when $i_s>1$ and $t< s$.

Define a relation set:
\eqqc{therelation}{
\arr{lcl}{\rho^{\bfg}(n) &= & \{ d_{s,t,\bfi} \gkk{t}{\bfi({s})} \gkk{s}{\bfi} - \gkk{s}{\bfi({t})} \gkk{t}{\bfi} \mid  \bfi,\bfi(t), \bfi({s}), \bfi(t)({s})\in Q(n)_0, 1\le t < s \le n\} \\
&& \qquad\qquad\cup \{\gkk{t}{\bfi(t)} \gkk{t}{\bfi} \mid  \bfi,\bfi({t}) \in Q(n)_0, 1\le t \le n\}.
}}
The relations are of two kind, one is zero relation consisting of arrows of same type, and the other is commutative relation consisting of arrows of two different types.

Let $\Lambda^{\bfg}(n)$ be the algebra with bound quiver $(Q(n ),\rho^{\bfg}(n))$, and we call it an {\em $n$-cubic pyramid algebra}.
$1$-cubic pyramid algebras are the quadratic dual of the hereditary algebras of quivers of type $A_m$ with linear orientation.

A path $p$ of a bound quiver $(Q, \rho)$ is called a bound path if its image in $kQ/(\rho)$ is non-zero.
Recall that a bound quiver $(Q, \rho)$ is called an {\em $n$-translation quiver} \cite{gu14} if there is a bijective map $\tau: Q_0 \setminus \mathcal P \longrightarrow Q_0 \setminus \mathcal I$, called the {\em $n$-translation} of $Q$, for two subsets $\mathcal P$ and $ \mathcal I $ of $Q_0$, whose elements are called {\em projective vertices} and respectively {\em injective vertices},  satisfying the following conditions:

\begin{enumerate}
\item[1.] Any maximal bound path is of length $n+1$ from  $\tau i$ in $Q_0 \setminus \mathcal I$ to $i$, for some vertex $i$ in $Q_0 \setminus \mathcal P $.

\item[2.] Two bound paths of length $n+1$ from $\tau i$ to $i$ are linearly dependent, for any $i \in Q_0 \setminus \mathcal P$.

\item[3.] For each $i \in Q_0 \setminus \mathcal P$ and $j\in Q_0$, any bound element $u$ which is linear combination of paths of the same length $t\le n+1$ from $j$ to $i$, there is a path $q$ of length $n+1-t$ from $\tau i$ to $j$ such that  $uq$ is a bound element.

\item[4.] For each $i \in Q_0 \setminus \mathcal I$ and $j\in Q_0$, any bound element $u$ which is linear combination of paths of the same length $t\le n+1$ from $i$ to $j$, there is a path $p$ of length $n+1-t$ from $j$ to $\tau^{-1} i$ such that  $pu$ is a bound element.
\end{enumerate}

An algebra $\Lambda$ with bound quiver an $n$-translation quiver $(Q,\rho)$ is called an {\em $n$-translation algebra} if there is an $q \in \mathbb N \cup \{\infty \}$ such that $\Lambda$ is $(n+1,q)$-Koszul in the following sense
\begin{enumerate}
\item $\Lambda_t=0$ for $t > n+1$, and

\item for each $i\in Q_0$, let

\eqqc{def_pqK}{ P^q(i) \slrw{f_q}\cdots \slrw{}P^1(i) \slrw{f_1}P^0(i) \slrw{f_0} S(i) = \LL_0 e_i \longrightarrow 0}
be the first $q+1$ terms in a minimal projective resolution of the simple $S(i) \simeq \LL_0 e_i$, then $P^t(i)$ is generated by its component of degree $t$ for $0\le t \le q$ and the kernel of $f_q$ is concentrated in degree $n+1+q$.
\end{enumerate}

\begin{theorem}\label{traa}
$(Q(n), \rho^{\bfg}(n))$ is an admissible $(n-1)$-translation quiver with $(n-1)$-translation defined by $\tau_{}: \bfi \lrw \bfi-\bfe_n$ if $i_n>1$.

$\LL^{\bfg}(n)$ is an $(n-1)$-translation algebra with admissible  $(n-1)$-translation quiver.
\end{theorem}

This theorem is a direct consequence of our main theorem,  Theorem ~\ref{main}.

Choosing $\tilde{\bfg}=(\tilde{g}_1, \ldots, \tilde{g}_n)$ such that $\tilde{g}_s(\gkk{t}{\bfi}) = -\gkk{t}{\tilde{g}_s(\bfi)}$ for $t< s$ and for $\bfi=(i_1,\ldots,i_n)$, when all the arrows are in $Q(n)$, the following proposition follows from Proposition ~\ref{compabsnalg}.

\begin{proposition}\label{abs}
The quadratic dual $\LL^{\tilde{\bfg},!}(n)$ of $\LL^{\tilde{\bfg}}(n)$ is the $n$-Auslander absolutely $n$-complete algebra $T_m\sk{n} (k)$.
\end{proposition}

It is easy to see that the set of the projective vertices in $Q(n)$  is $\mathcal P=\{\bfi\mid  i_n=1\}$, and the set of the injective vertices is $\mathcal I=\{\bfi\mid \,\, |\bfi|= n + m - 1\}.$

Let $m \ge 3$ and let $Q(n)$ be an $n$-cubic pyramid quiver.
Define {\em stable $n$-cubic pyramid  quiver} $\tQ(n)$ of height $m$ as the quiver with \eqqc{stblequiver}{\arr{ll}{\mbox{vertex set:} & \tQ(n)_0=Q(n)_0\\ \mbox{arrow set:} & \tQ(n)_1=Q(n)_1\cup \{\gkk{n+1}{\bfi}:\bfi \lrw \bfi-\bfe_n\mid  i_n>1 \}.}}
The stable quiver of our earlier example is as follow:
{\tiny 
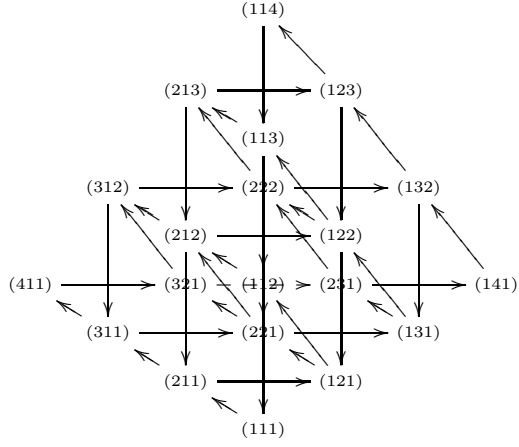
$$
{\xymatrix@C=0.2cm@R0.2cm{
&&&&&&&&&&&&&&&\\
&&&&&&&&&&&&&&&\\
& &&(114)\ar[ddd]\\
&&&&&&&&&&&&&&&\\
&& (213)\ar[rr]\ar[ddd]&&(123)\ar[ddd]\ar[uul]\\
&&&(113)\ar[ul]\ar[ddd]&&&&\\
&(312)\ar[ddd]\ar[rr]&& (222)\ar[ddd]\ar[uul]\ar[rr]&&(132)\ar[ddd]\ar[uul]\\
&&(212)\ar[ddd]\ar[rr]\ar[ul]&&(122)\ar[ddd]\ar[uul]\ar[ul]&&&&\\
(411)\ar[rr]&&(321)\ar[uul]\ar@{-->}[rr]&(112)\ar[ddd]\ar[ul]&(231)\ar[uul]\ar[rr]&&(141)\ar[uul]&& \\
&(311)\ar[ul]\ar[rr]&&(221)\ar[rr]\ar[ul]\ar[uul]&&(131)\ar[uul]\ar[ul]&&\\
&&(211)\ar[ul]\ar[rr]&&(121)\ar[ul]\ar[uul]&&&&&&&\\
&&&(111)\ar[ul]&&&&
}}$$}
It follows from \cite{g13} and \cite{gcv} that such quiver is a truncation of the McKay quiver of some finite Abelian subgroup of $SL_{n+1}(\mathbb C)$.

Let $\bbfg=(\bfg,g')$ for some graded endomorphism $g'$ induced by the $(n-1)$-translation defined in Theorem ~\ref{traa}, then $g'(\gkk{s}{\bfi-\bfe_n}) = d_{n+1,s,\bfi-\bfe_n}\gkk{s}{\bfi-\bfe_n} $ for some $0 \neq d_{n+1,s,\bfi-\bfe_n} \in k$, since there is at most one arrow of each type $s$ from any vertex of $Q(n)$.
Define a relation set
\eqqc{stablerelation}{\arr{rcl}{\tilde{\rho}^{\bbfg}(n) &=&\rho^{\bfg}\cup\{\gkk{n+1}{\bfi-\bfe_n}\gkk{n+1}{\bfi}\mid i_n > 1\} \\&&
\cup\{ d_{n+1,s,\bfi-\bfe_n}\gkk{s}{\bfi-\bfe_n}\gkk{n+1}{\bfi} - \gkk{n+1}{\bfi-\bfe_{s-1}+\bfe_s}\gkk{s}{\bfi} \mid  \bfi\in \tQ(n)_0, \\ && \qquad i_n > 1,i_{s-1}>1 \mbox{ and } |\bfi|\le m+n-1 \}
}}

Let  $\tL^{\bbfg}(n)$ be the algebra given by the stable pyramid $n$-cubic  quiver $\tQ(n)$ and the relation $\tilde{\rho}^{\bbfg}(n)$.
Our key results in this paper, Lemma ~\ref{hproofof}, essentially proves the following theorem.
\begin{theorem}\label{trata}
$\tL^{\bbfg}(n)$ is an $n$-translation algebra with stable  $n$-translation quiver and trivial $n$-translation.
\end{theorem}
Set $\mR = \{1,2,\ldots,n\}$ and $\mRp = \{1,2,\ldots,n+1\}$.
Write $$\gkk{t_1,\ldots,t_{s}}{\bvi{\bfa}} = \gkk{t_s}{\bvi{\bfa+\bfe_{t_1,\ldots,t_{s-1}}}} \ldots \gkk{t_1}{\bvi{\bfa}} $$ for a path determined by a sequences $\{t_1,\ldots,t_s\}\subset \mRp$.
For a subset $P=\{p_1,\ldots,p_{s}\}$ of $ \mRp$ with $1\le p_1< \cdots < p_s\le n+1 $, write $$\gkk{P}{\bvi{\bfa}} = \gkk{p_1,\ldots,p_{s}}{\bvi{\bfa}}.$$

The following lemma describing the bound paths in $\tQ(n)$.
\begin{lemma}\label{pathlengthl}
Assume that $0\le l \le n+1$.
For any two vertices $\bfi, \bfj\in \tQ(n)_0$, we have that $\dmn_k e_{\bfj} \tL^{\bbfg}(n)_l e_{\bfi} \le 1$, and $\dmn_k e_{\bfj} \tL^{\bbfg}(n)_l e_{\bfi} = 1$ if and only if any path of length $l$ from $\bfi$ to $\bfj$ consists of arrows of different types.
\end{lemma}
\begin{proof}
Write  $\bar{p}$ for the image of a path $p$ in $\tL^{\bbfg}(n)$.
Let $p$ be a path of length $l$ in $\tQ(n)$ from  $\bfi$ to $\bfj$, then $p= \gkk{t_l}{\bfi(t_1, \ldots, t_{1-1} )} \cdots \gkk{t_1}{\bfi}$.
This path ends at $\bfj_j = \bfi( t_1, \ldots, t_l)$.
For any path $p'  = \gkk{t'_l}{\bfi(t'_1, \ldots, t'_{1-1} )} \cdots \gkk{t_1'}{\bfi}$  of length $l$ from  $\bfi$ to $\bfj$, we have that $t'_1, \ldots,t'_l$ a permutation of $t_1, \ldots, t_l$, since $\bfj_j = \bfi( t'_1, \ldots, t'_l)$.
Now we prove by induction that $\bar{p}$  is a multiple of  $\bar{p'}$.
If $t_1=t'_1$, then  by induction $q= \gkk{t_l}{\bfi(t_1, \ldots, t_{1-1} )} \cdots \gkk{t_2}{\bfi(t_1)}$ and $q'=\gkk{t'_l}{\bfi(t'_1, \ldots, t'_{1-1} )} \cdots \gkk{t_2}{\bfi(t_1)}$ are paths of length $l$ in $\tQ(n)$ from  $\bfi(t_1)$ to $\bfj$, and $\bar{q}$  is a multiple of  $\bar{q'}$.
Thus $\bar{p}$  is a multiple of   $\bar{p'}$.

Otherwise, $\bfi(t_1), \bfi(t'_1) \in \tQ(n)_0$, thus $\bfi(t_1,t'_1) \in \tQ(n)_1$ and $\gkk{t'_1}{\bfi(t_1)}, \gkk{t'_1}{\bfi(t_1)} \in \tQ(n)_1$.
Assume that $t_1'=t_s$, then we have $\bfi(t'_1, t_1), \ldots, \bfi(t'_1, t_1, \ldots, t_{s-1} )\in \tQ(n)_0 $, and  $\overline{ \gkk{t_{s}}{ \bfi( t_1, \ldots, t_{s-1} )} \ldots \gkk{t_1}{\bfi}}$  is a multiple of $\overline{ \gkk{t_{s-1}}{ \bfi( t_1', t_1, \ldots, t_{s-2} )}\ldots  \gkk{t_1}{\bfi(t'_1)}\gkk{t'_1}{\bfi}}$,  using the commutative relations in \eqref{stablerelation}.
So in the algebra $\tL^{\bbfg}(n)$, we have that  $\bar{p}$  is a multiple of  $\overline{ \gkk{t_{l}}{\bfi(t_1,\ldots,t_{l-1})}\cdots \gkk{t_{s+1}}{\bfi(t_1,\ldots,t_s)}\gkk{t_{s-1}}{ \bfi( t_1', t_1, \ldots, t_{s-2} )}\ldots  \gkk{t_1}{\bfi(t'_1)}\gkk{t'_1}{\bfi}}$.
By inductive assumption $\bar{p'}$  is a multiple of  $\overline{ \gkk{t_{l}}{\bfi(t_1,\ldots,t_{l-1})}\cdots \gkk{t_{s+1}}{\bfi(t_1,\ldots,t_s)}\gkk{t_{s-1}}{ \bfi( t_1', t_1, \ldots, t_{s-2} )}\ldots  \gkk{t_1}{\bfi(t'_1)}\gkk{t'_1}{\bfi}}$.
Thus $\bar{p}$ is a multiple of $\bar{p'}$.

Similarly, using both  commutative relations and zero relations in \eqref{stablerelation}, we see that $\bar{p}$ is zero in $\tL^{\bbfg}(n)_l$ whenever two of $t_1, \ldots, t_l$ are equal.

By comparing dimensions of the corresponding spaces in $k\tQ(n)$ and in $(\tilde{\rho}^{\bbfg}(n))$, the ideal generated by $\tilde{\rho}^{\bbfg}(n)$, we see that if $t_1, \ldots, t_{1-1}, t_l$ are pairwise different, its image  $\bar{p}\neq 0$ in $\tL^{\bbfg}(n)_l$, and $\dmn_k e_{\bfj}\tL^{\bbfg}(n)_l e_{\bfi} =1$.

Note that the only cyclic bound paths in $\tQ(n)$ are of length $n+1$.

This proves our lemma.
\end{proof}

So we see that a bound path starting at a vertex $\bfi$ in $\tQ(n)$ is a path on $H^{\bfi}$ starting from $\bfi$.
Since the $n$-translation is defined by $\tau \bfi = \bfi-\bfe_n$ for each non-projective vertex $\bfi$ in $Q(n)_0$, that is, the ones with $\bfi_n >1$.

The following theorem follows from Theorem ~\ref{main} and  Proposition 4.2 of \cite{gu14}.
\begin{theorem}\label{twistext} $\tL^{\bbfg}(n)$ is isomorphic to a twisted trivial extension of $\LL(n)$.
\end{theorem}

For any $\bfi\in Q(n)_0$, let $S(\bfi)$ be the simple $\LL(n)$-module at $\bfi$ and let $P(\bfi)$ and $I(\bfi)$ be its projective cover and injective envelope as $\LL(n)$-modules, respectively.
Let $\tS(\bfi)$ be the simple $\tL(n)$-module at $\bfi$ and let $\tP(\bfi)$ and $\tI(\bfi)$ be its projective cover and injective envelop as $\tL(n)$-modules, respectively.
Conventionally,  set $\tS(\bfi)$, $\tP(\bfi)$ and $\tI(\bfi)$ to be zero when $\bfi \not \in \tQ(n)_0$.

The supports of $\tP(\bfi)$ and $\tI(\bfi)$ are described with $U\sk{n+1}$, we have the following lemma.

\begin{lemma}\label{proj}
$$\bfr^t \tP(\bfi)/\bfr^{t+1} \tP(\bfi) \simeq \bigoplus_{\bfa \in U_t\sk{n+1}} \tS(\bvi{\bfa}).$$
$$\soc^{t+1} \tP(\bfi)/\soc^{t} \tP(\bfi) \simeq \bigoplus_{\bfa \in U_{n-t}\sk{n+1}} \tS(\bfv^{\bfi-\bfe_{n}}(\bfe-\bfa))$$
\end{lemma}
\begin{proof}
Since $\tL^{\bbfg}(n)$ is a graded algebra, $ \bfr^t \tP(\bfi)/\bfr^{t+1} \tP(\bfi) \simeq \tP(\bfi)_{t}$ as  $\tL(n)_0$ modules.
So $\tP(\bfi)_{t}$ has a basis consists of the images of the paths of length $t$, and a non-zero path of length $t$ from $\bfi$ to $\bfj$ generates a submodule isomorphic to $\tS(\bfj)$.
Let $I_t$ be the set of the ending vertices of the paths of length $t$ from $\bfi$, it follows from Lemma ~\ref{pathlengthl} that $ \bfr^t \tP(\bfi)/\bfr^{t+1} \tP(\bfi) \simeq \bigoplus_{\bfj \in I_t} \tS(\bfj)$.
A vertex $\bfj$ in $I_t$ is the ending vertex of s  path of length $t$ in $Q$ passing through $t$ arrows of different types,  say $u_1,\ldots,u_t$, thus $\bfj = \bfi(u_1,\ldots,u_t) = \bvi{\bfa}$, for $\bfa = \bfe_{u_1}+ \cdots + \bfe_{u_t}$ in $\tQ(n)$.
Clearly  $\bfa \in U_t$.
By our assumption, $\tS(\bvi{\bfa})=0$ if $\bvi{\bfa}$ is not a vertex of $\tQ(n)$, thus $$\bfr^t \tP(\bfi)/\bfr^{t+1} \tP(\bfi) \simeq \bigoplus_{\bfa \in U_t\sk{n+1}} \tS(\bvi{\bfa}).$$

Note that $\tP(\bfi) = \tI(\bfi-\bfe)$, the second assertion follows from the dual argument.
\end{proof}

Since the components of $\bfa$ are $0$ or $1$, it is easy to see that $\bvi{\bfa}$ is a quiver vertex if and only if the following hold for all $s$: (i). $a_{s+1} \le a_s$ if $i_s=1$; (ii). $a_{s+1} \ge a_s$ if $\sum\limits_{t=1}^s i_t =m+s-1$.
A vertex $\bfi$ with $i_s=1$ or $\sum\limits_{t=1}^s i_t =m+s-1$ for some $s$ is called a {em boundary vertex}, otherwise it is called an {\em internal vertex}.
Now let $Z(\bfi)= \{s\in \mR \mid  i_s=1\}$ and $W(\bfi)=\{s\in \mR\mid  \sum\limits_{t=1}^s i_t =m+s-1\}$.

The set $U\sk{n+1}(\bfi)$ of $\bfi$-quiver vertices in $U\sk{n+1}$ is
$$U\sk{n+1}(\bfi) =\{\bfa=(a_1,\ldots,a_{n+1})\mid a_s \ge a_{s+1}\mbox{ if } s\in Z(\bfi), a_s \le a_{s+1}\mbox{ if } s\in W(\bfi) \}.
$$
This implies that if $i_s=1 $ and $a_{s+1}=1$, then $a_s =1$, and if $s\in W(\bfi)$ and $a_{s}=1$, then $a_{s+1}=1$.
We can refine Lemma ~\ref{proj} as following

\begin{lemma}\label{projrefine}
$$\bfr^t \tP(\bfi)/\bfr^{t+1} \tP(\bfi) \simeq \bigoplus_{\bfa \in U_t\sk{n+1}(\bfi)} \tS(\bvi{\bfa}).$$
$$\soc^{t+1} \tP(\bfi)/\\soc^{t} \tP(\bfi) \simeq \bigoplus_{\bfa \in U_{n-t}\sk{n+1}(\bfi-\bfe_n)} \tS(\bfv^{\bfi-\bfe_{n}}(\bfe-\bfa))$$
\end{lemma}

Now consider $\olQ(n)=\mathbb Z|_{(n-1)}Q(n)$ and $0$-extension $\olL(n)^{\bbfg}= \tL^{\bbfg}(n)\#\mathbb Z^*$, respectively (See \cite{gu14}).
We have
\eqqc{barquiver}{\arr{lll}{\olQ(n)_0 & = & Q(n)_0\times \mathbb Z\\
\olQ(n)_1 & = & Q(n)_1\times \mathbb Z\cup \{\gkk{n+1}{\bfi,v}:(\bfi,v) \lrw (\bfi-\bfe_n,v+1) \\ && \qquad\qquad \qquad \mid (\bfi,v), (\bfi-\bfe_n,v+1)\in \olQ(n)_0 \}.}}
Set $\bfe\sk{n+1}_t =(\bfe_t\sk{n},0)$ for $1\le t\le n$ and $\bfe\sk{n+1}_{n+1} =(\mathbf 0\sk{n},1)$, then the arrows are $\gkk{t}{\bfi}: \bfi \lrw \bfi-\bfe_{t-1}+\bfe_t $ for $\bfi\in \olQ(n)_0$, $1\le t\le n+1$.
And the relations for $\olL^{\bbfg}(n)$ is induced from $\tilde{\rho}^{\bbfg}(n)$.
Thus
\eqqc{barrelation}{\arr{lcl}{\bar{\rho}^{\bbfg}(n) &= &\rho^{\bfg}\times \mathbb Z
\cup \{\gkk{n+1}{\bfi-\bfe_n+\bfe_{n+1}}\gkk{n+1}{\bfi}\mid i_n > 1\} \\&& \cup\{ d_{n+1,s,\bfi-\bfe_n+\bfe_{n+1}}\gkk{s}{\bfi-\bfe_n}\gkk{n+1}{\bfi} - \gkk{n+1}{\bfi-\bfe_{s-1}+\bfe_s}\gkk{s}{\bfi} \mid   \\ && \qquad   \bfi\in \olQ(n)_0,i_n > 1,i_{s-1}>1 \mbox{ and } \sum\limits_{t=1}^s i_t < m+s-1, 1\le s\le n\}.}}

The following theorem follows from Theorem ~\ref{main} 
and Theorem 5.3 of \cite{gu14}.

\begin{theorem}\label{traba}
$\olL^{\bbfg}(n)$ is an $n$-translation algebra with stable  $n$-translation quiver $\olQ(n)$ and $n$-translation $\otau_{n}: \bfi = \bfi-\bfe_{n+1}$.
\end{theorem}

Let $(Q,\rho)$ be a bound quiver and $Q'$ a full subquiver of $Q$.
Let $\rho[Q']= \{e_j\sum a_p p e_i\mid  \sum a_p p \in \rho, i,j \in Q'_0\}$, we call $(Q', \rho[Q'])$ a {\em full bound subquiver of $Q$} if for any $i, j \in Q'_0$, and for any path $p$ in $Q$ with $a_p \neq 0$ for some $\sum a_p p \in \rho$ and $e_j p e_i \neq 0$, then $p$ is a path in $Q'$.
In this case, we say $\rho'=\rho[Q']$ is a {\em restriction} of $\rho$ on $Q'$.

We have the following lemma.

\begin{lemma}\label{subbdquiver}
Let $\Lambda=kQ/(\rho)$ be the algebra with bound quiver $(Q,\rho)$ and let $(Q',\rho')$ be a full bound subquiver of $(Q,\rho)$.
Let $I$ be the ideal generated by the set of idempotents $\{e_j \mid  j\in Q_0\setminus Q'_0\}$.
Then $$kQ'/(\rho') \simeq \LL/I.$$
\end{lemma}
\begin{proof}
Let $\phi: kQ \to kQ'$ be the homomorphism which is identity on $Q'$ and sending paths not inside $Q'$ to zero, then $\Ker \phi =(E) $ is generated by the set $E=\{ e_j | j \in Q_0\setminus Q'_0\}$ of idempotents corresponding to the vertices outside $Q'$, and $kQ'\simeq kQ/(E)$.
Let $\pi : kQ \to \LL $ be the canonical homomorphism, then $\Ker \pi =(\rho)$ and $\LL = kQ/(\rho)$.
Let $\pi': kQ' \to \Lambda'$ be the canonical homomorphism with  $\Ker \pi =(\rho')$,  then $\LL= kQ'/(\rho')$.
$\pi' \phi: kQ \to \Lambda'$ is an epimorphism and we have that $\Ker \pi' \phi$ is generated by  $E\cup \rho'$.
On the other hand $(E, \rho') = (E, \{a_{Q' } | a \in \rho\}) = (E, \rho)$, since $\rho'$ is a restriction of $\rho$.
So $(\rho)\subset \Ker \pi \phi $ and the map $\pi' \phi = \psi \pi $ factor through the canonical homomorphism $\pi: kQ \to kQ/(\rho) = \Lambda$.
Thus $\psi : \Lambda\to \Lambda'$ is an epimorphism and $\Ker \psi = \pi{E} =(E)$, when identifying the idempotents in $kQ$ and in $\LL$.
That is $\Lambda' \cong \Lambda/(E)$.
\end{proof}

\section{Minimal Projective Resolutions of the Simples of $\tL^{\bbfg}(n)$}

In this section, we study the projective resolution of the simples of the algebra $\tL^{\bbfg}(n)$.
Fix $\bfi \in \tQ(n)_0$, we have a vector $\bfb = \bfb(\bfi) =(b_1,\ldots, b_{n+1})$ in \eqref{defineb} which defines a cuboid $C\sk{n}(\bfi)$ in \eqref{defineC} for $\bfi$.
The following lemma follows from the definition of $\bvi{\bfa}$ for $\bfa \in \Zp^{n+1}$.

\begin{lemma}\label{Clnomul} For $0\le l\le m-1$, if $\bfa, \bfa'\in \Zp^{n+1}_l $, then $\bvi{\bfa} =\bvi{\bfa'} $  if and only if $\bfa=\bfa'$.
\end{lemma}

For $\bfa \in \Zp^{n+1}$.
Let $O(\bfa)= \{t\mid a_t=0\}$, $T(\bfa)=\{1\le t \le n+1\mid  a_t > b_t\}$ and let $\mR(\bfa)= \mRp \setminus (O(\bfa)\cup T(\bfa))$. 
Then  $\bfa \in C\sk{n}(\bfi)$ if and only if $T(\bfa)=\emptyset$.

Let \eqqc{hprojres}{\cdots \lrw \tP^2(\tS(\bfi))\stackrel{f_2}{\lrw} \tP^1(\tS(\bfi))\stackrel{f_1}{\lrw} \tP^0(\tS(\bfi))\stackrel{f_0}{\lrw}  \tS(\bfi) \lrw 0 }
be a minimal projective resolution of the simple $\tL$-modules $\tS(\bfi)$ corresponding to the vertex $\bfi\in \tQ(n)_0$.

Our main aim of this section is characterizing  the linear part of this projective resolution using $(n+1)$-cuboid $C\sk{n}(\bfi)$ as in the following proposition.
The technical detail of the proof will be given in Lemma ~\ref{hproofof}.
For the last assertion, notice the fact that $\bvi{\bfa+\bfe}=\bvi{\bfa}$.
\begin{proposition}\label{pprojresta} For $l=0, 1, \ldots, m-1$
\eqqc{hprtml}{\tP^l(\tS(\bfi)) \simeq \bigoplus_{\bfa\in C\sk{n}_l({\bfi})} \tP(\bvi{\bfa}),}
and $\Ker f_{m-1} \simeq \tS(\bvi{\bfb(\bfi)})$.
\end{proposition}

Let $\{\ep{l}{\bvi{\bfa}} \mid  \bfa\in C\sk{n}_l({\bfi})\}$ be the standard basis of the direct sum \eqref{hprtml}, then $$e_{\bvi{\bfa}} \ep{l}{\bvi{\bfa}} = \ep{l}{\bvi{\bfa}} \mbox{ and } e_{\bfj} \ep{l}{\bvi{\bfa}} =0 \mbox{ if } \bfj \neq \bvi{\bfa}.$$
We assume  that $ \ep{l}{\bvi{\bfa}}=0$  when $\bfa \not\in  C\sk{n}_l(\bfi)$ or $\bvi{\bfa}\not\in \tQ(n)_0$, for $l=1,\ldots, m-1$, conventionally.



Let $\bfa' \in \Zp^{n+1}$.
Consider the composition factor $\tS(\bvi{\bfa'})$ in the degree $t$ component of \eqref{hprojres}, we have the following lemma.

\begin{lemma}\label{compositionfactor}
The multiplicity of the composition factor $\tS(\bvi{\bfa'})$ in  $\bigoplus_{\bfa\in C\sk{n}_s({\bfi})} \tP(\bvi{\bfa})_{t}$ is $\cmb{|\mR(\bfa')|}{t-|T(\bfa')|}$ if $|T(\bfa')| \le t$ and is $0$ if $|T(\bfa')| > t$.

\end{lemma}
\begin{proof}
The composition factor $\tS(\bfa')$ in the degree $t$ component  $\bigoplus_{\bfa\in C\sk{n}_s({\bfi})} \tP(\bvi{\bfa})_{t}$ is
\eqqc{hprojresl}{\bigoplus_{\bfa\in C\sk{n}_{s}(\bfi) }e_{\bvi{\bfa'}} (\tP(\bvi{\bfa}))_{t} .}
Note that \small$$\bigoplus_{\bfa\in C\sk{n}_{s}(\bfi) } e_{\bvi{\bfa'}} (\tP(\bvi{\bfa}))_{t} =\bigoplus_{\bfa\in C\sk{n}_{s}(\bfi) }e_{\bvi{\bfa'}} \bfr^{t}\tL(n)/\bfr^{t+1}\tL(n) e_{\bvi{\bfa}} $$\normalsize
is the space spanned by the bound paths of length $t$ ending at $\bvi{\bfa'}$.
The multiplicity of the composition factor $\tS(\bvi{\bfa'})$ in  $\bigoplus_{\bfa\in C\sk{n}_s({\bfi})} \tP(\bvi{\bfa})_{t}$ is exactly
$$\dmn_k \bigoplus_{\bfa\in C\sk{n}_{s}(\bfi) } e_{\bvi{\bfa'}} (\tP(\bvi{\bfa}))_{t},$$
that is, the order of the maximal set of linearly independent paths of length $t$ from a vertex $\bvi{\bfa}$ with $\bfa $ in $C\sk{n}_s(\bfi)$ to $\bvi{\bfa'}$.

Note that $T(\bfa)=\emptyset $ for $\bfa \in C\sk{n}(\bfi)$.
If $t < |T(\bfa')|$, the length of any path from $\bvi{\bfa}$ to $\bvi{\bfa'}$ is larger than $|T(\bfa')|$, this implies that  $\tS(\bvi{\bfa'})$ is not a composition factor of  $\bigoplus_{\bfa\in C\sk{n}_s({\bfi})} \tP(\bvi{\bfa})_{t}$.

Assume that there is a bound paths of length $t$ from $\bvi{\bfa}$ to $\bvi{\bfa'}$, with  $\bfa \in C\sk{n}_s(\bfi)$.
Then $t \ge |T(\bfa')|$, and $O(\bfa') \subset O(\bfa)$.
By Lemma ~\ref{pathlengthl}, any two such paths are linearly dependent.
Let $S(\bfa',\bfa)= \{ z\in \mRp | b_z\ge a'_z>a_z \}\subset \mR(\bfa')$, then $|S(\bfa',\bfa)|= t-|T(\bfa')|$ and $a'_z=a_z+1$ for $z\in T(\bfa)\cup S(\bfa',\bfa)$.
The types of the arrows appearing in any bound path  of length $t$ from $\bvi{\bfa}$ to $\bvi{\bfa'}$ form the set $T(\bfa') \cup S(\bfa', \bfa ) $, we may take the path $\gkk{T(\bfa')}{\bfa+\bfe_{S (\bfa', \bfa)}} \gkk{S(\bfa', \bfa)}{\bfa}$ as the representative of the bound paths  of length $t$ from $\bvi{\bfa}$ to $\bvi{\bfa'}$.

On the other hand, for each subset $S\subset \mR(\bfa') $ with $|S|=t-|T(\bfa')|$, let $\bfa=\bfa' -\bfe_{S \cup T(\bfa')}$, then $T(\bfa)=\emptyset$ and $\bfa \in C\sk{n}(\bfi)$.
In this case, $\gkk{T(\bfa')}{\bfa+\bfe_{S}} \gkk{S}{\bfa}$ is a bound paths  of length $t$ from $\bvi{\bfa}$ to $\bvi{\bfa'}$.

Obviously, the condition that $\bfa \in C\sk{n}_s(\bfi)$ is equivalent to that $s+t-|T(\bfa')| \le m-1$.

So we see that there are $\cmb{|\mR(\bfa')|}{t-|T(\bfa')|}$ linearly independent paths from some vertex $\bvi{\bfa}$ with $\bfa \in C\sk{n}(\bfi)$ to $\bvi{\bfa'}$.
\end{proof}

With the convention that $\cmb{\mR(\bfa')|-1}{-1} =0$, and $\cmb{\mR(\bfa')|-1}{0} =1$, we have the following lemma.

\begin{lemma}\label{kercompositionfactor}
Assume that for $0\le l \le s $, $\tP^l(\tS(\bfi))$ in \eqref{hprojres} is generated in degree $l$ and  \eqref{hprtml} holds.
Then the multiplicity of the composition factor $\tS(\bvi{\bfa'})$ in degree $t$ part of $\Ker f_{s}$  is $\cmb{|\mR(\bfa')|-1}{t-s-|T(\bfa')|-1}$, especially, it is $0$ if $t-s-|T(\bfa')|=0$ and is $1$ if $t-s-|T(\bfa')|=1$.
\end{lemma}
\begin{proof}
Since \eqref{hprtml} holds for $0\le l \le s $,  the composition factor $\tS(\bfa')$ in the degree $t$ components of the minimal  projective resolution \eqref{hprojres} up to $s+1$ term is
\eqqc{hprojres2}{\arr{lll}{0 &\lrw&  e_{\bvi{\bfa'}}(\Ker f_s)_{t} \lrw \bigoplus_{\bfa\in C\sk{n}_{s}(\bfi) }e_{\bvi{\bfa'}} (\tP(\bvi{\bfa})[s])_{t} \\ & &\lrw \bigoplus_{\bfa\in C\sk{n}_{s-1}(\bfi) }e_{\bvi{\bfa'}} (\tP(\bvi{\bfa})[s-1])_{t} \lrw \cdots \\ & & \lrw \bigoplus_{\bfa\in C\sk{n}_{s'}(\bfi) } e_{\bvi{\bfa'}} (\tP(\bvi{\bfa})[s-s'])_{t}  \lrw 0,}}
with $s' = |\mR(\bfa')|- t+s$.

We have that $$\arr{lll}{&& \dmn_k\bigoplus_{\bfa\in C\sk{n}_{s-h}(\bfi) } e_{\bvi{\bfa'}} (\tP(\bvi{\bfa})[s-h])_{t} \\&=& \dmn_k \bigoplus_{\bfa\in C\sk{n}_{s-h}(\bfi) } e_{\bvi{\bfa'}} (\tP(\bvi{\bfa}))_{t-s+h}\\ &=& \cmb{|\mR(\bfa')|}{t-s+h-|T(\bfa')|}},$$
by Lemma ~\ref{compositionfactor}.
Thus
$$
\arr{lcl}{&& \dmn_k e_{\bvi{\bfa'}}(\Ker f_s)_{t} \\ & = & \sum\limits_{h=0}^{|\mR(\bfa')|-t+s} (-1)^{h} \sum\limits_{\bfa\in C\sk{n}_{s-h}(\bfi) } \dmn_k e_{\bvi{\bfa'}}(\tP(\bvi{\bfa})[s-h])_{t}\\
&=&\sum\limits_{t' = t-s- |T(\bfa')|}^{|\mR(\bfa')|} (-1)^{h}\cmb{|\mR(\bfa')|}{t'}\\
&=&\cmb{|\mR(\bfa')|-1}{t-s-|T(\bfa')|-1}.
}$$
If $t-s-|T(\bfa')|-1=0$, that is $t-s-|T(\bfa')|=1$, then $$ \sum\limits_{t' = 1}^{|\mR(\bfa')|} (-1)^{t'-1}\cmb{|\mR(\bfa')|}{t'}  =  1+(-\sum\limits_{t'=0}^{|\mR(\bfa')|} (-1)^{t'}\cmb{|\mR(\bfa')|}{t'} )\\ = 1+(1-1)^{|\mR(\bfa')|} =1.$$
If $t-s-|T(\bfa')|-1=-1$, $t-s-|T(\bfa')|=0$, so $$\sum\limits_{t' = 0}^{|\mR(\bfa')|} (-1)^{t'}\cmb{|\mR(\bfa')|}{t'} =(1-1)^{|\mR(\bfa')|} =0.$$
\end{proof}

\begin{lemma}\label{hproofof}
In the projective resolution \eqref{hprojres}, we have
\begin{enumerate}
\item[(a)] \eqref{hprtml} hold for $0\le l \le m-1$;

\item[(b)] for $1\le l \le m-1$, and for each $\bfa\in C\sk{n}_{l}(\bfi)$, there are elements $\zk{l-1}{\bfa-\bfe_t}\in k$ for $t=1,\ldots,n+1$, with $\zk{l-1}{\bfa- \bfe_t} \neq 0$ if $\bfa-\bfe_t \in C\sk{n}_{l-1}(\bfi)$, and $\zk{l-1}{\bfa- \bfe_t} =0$ otherwise, such that the element $\theta^{(l-1)}_{\bvi{\bfa}} = \sum\limits_{t=1}^{n+1} \zk{l-1}{\bfv^{\bfi}(\bfa-\bfe_t)} \gkk{t}{\bfv^{\bfi}(\bfa-\bfe_t)} \ep{l}{\bfv^{\bfi}(\bfa-\bfe_t)}$ satisfies  \eqqc{hkerlsummand}{e_{\bvi{\bfa}}\ta{l-1}{\bvi{\bfa}} = \ta{l-1}{\bvi{\bfa}}}  and the map $f_l$ is defined by \eqqc{hprmapl}{f_{l}(\ep{l}{\bvi{\bfa}}) = \theta^{(l-1)}_{\bvi{\bfa}};}

\item[(c)] \eqqc{hkerlre}{K_l=\{\ta{l}{\bvi{\bfa}} \mid  \bfa\in C\sk{n}_{l+1}(\bfi) \} } generates $\Ker f_l$ for $0\le l \le m-2$;

\item[(d)] $C\sk{n}_{m-1}({\bfi}) = \{ \mathbf b \}$, $\mathbf b +\bfe$ is a $\bfi$-quiver vertex and $\Ker f_{m-1} \simeq S( \bvi{\mathbf b+\bfe})$.
\end{enumerate}
\end{lemma}

\begin{proof}
We now prove (a), (b) and (c) using induction on $l$.
The assertions for $\tP^0(S(\bfi)),  \Ker f_0$ and $f_1$ are obvious.

Assume that $1\le l\le  m-1$ and the assertions hold for $\tP^{h}(\bfi), \Ker f_{h}$, and $f_{h}$ for $h<l$.
By the inductive assumption, $\Ker f_{l-1}$ is a graded module generated in degree $l$ and as $k$-spaces
$$(\Ker f_{l-1})_{l} = \mathrm{top}\,(\Ker f_{l-1}) =L(\ta{l-1}{\bvi{\bfa}} \mid  \bfa\in C\sk{n}_{l}(\bfi)) =\sum\limits_{\bfa\in C\sk{n}_{l}(\bfi)} k\ta{l-1}{\bvi{\bfa}}.$$

By \eqref{hkerlsummand}, we have
\eqqc{htopom}{ k\ta{l-1}{\bvi{\bfa}}\simeq \tL_0 e_{\bvi{\bfa}} = \tS(\bvi{\bfa}),}
as $\tL_0$-modules.
Since for each $ \bfa \in C\sk{n}_{l}(\bfi)$, we have  exactly  one non-zero  element $ \ta{l-1}{\bvi{\bfa}} $ in $K_{l-1}$.
Thus $$\mathrm{top}\,(\Ker f_{l-1}) \simeq \bigoplus_{\bfa\in C\sk{n}_{l}(\bfi)} \tL_0 e_{\bvi{\bfa}}.$$
So $$\tP^{l}(S(\bfi)) \simeq \bigoplus_{\bfa\in C\sk{n}_{l}(\bfi)} \tP(\bvi{\bfa}).$$
This proves (a) for $l$.

Define the homomorphism $f_l$ from $\tP^{l}(S(\bfi))$ to $\tP^{l-1}(S(\bfi))$ with
$$f_l(\ep{l}{\bvi{\bfa}})= \ta{l-1}{\bvi{\bfa}},
$$
Clearly $f_l$ is an epimorphism from $\tP^{l}(S(\bfi))$ to $\Ker f_{l-1}$, and $f_l (\tP^{l}(S(\bfi)))_p \subseteq (\tP^{l-1}(S(\bfi)))_{p+1}$ for $p \in \mathbb Z$.
This proves \eqref{hprmapl} for $l$.

Now we compute generators of  $\Ker f_l$ for $l \le m-2$.
Clearly  $(\Ker f_l)_l =0$ and $(\Ker f_l)_{l+n+1} =\soc\tP^l (S(\bfi))$, by the inductive assumption.

Assume that $0\neq x_{p}$ is a homogeneous element in $(\Ker f_l)_{p}$, then we can write $ x_{l+1} = \sum\limits_{\bfa\in C\sk{n}_{l}(\bfi)}(\sum\limits_{t=1}^{n+1} \zks{\bfa,t}\gkk{t}{\bvi{\bfa}} )\ep{l}{\bvi{\bfa}} $, with the convention that $\zks{\bfa,t}=0$ provided that $\bvi{\bfa+\bfe_t} \not\in\tQ(n)_0$.
Thus \eqqc{hkereq}{\arr{ccl}{
0&= &f_l(x_{l+1})  = \sum\limits_{\bfa\in C\sk{n}_{l}(\bfi)} (\sum\limits_{t=1}^{n+1} \zks{\bfa,t}\gkk{t}{\bvi{\bfa}}) f_l(\ep{l}{\bvi{\bfa}})\\
&=& \sum\limits_{\bfa\in C\sk{n}_{l}(\bfi)} (\sum\limits_{t=1}^{n+1} \zks{\bfa,t}\gkk{t}{\bvi{\bfa}}) \cdot \ta{l-1}{\bvi{\bfa}}\\
&=& \sum\limits_{\bfa\in C\sk{n}_{l}(\bfi)} (\sum\limits_{t=1}^{n+1} \zks{\bfa,t}\gkk{t}{\bvi{\bfa}}) \cdot (
\sum\limits_{t=1}^{n+1}\zk{l-1}{\bfa-\bfe_t}\gkk{t}{\bfv^{\bfi}(\bfa-\bfe_t)} \ep{l-1}{\bfv^{\bfi}(\bfa-\bfe_t)})\\
&=& \sum\limits_{\bfa\in C\sk{n}_{l}(\bfi)} ( \sum\limits_{\stackrel{t'\neq t}{1\le t,t'\le n+1 }} \zks{\bfa,t}\zk{l-1}{\bfa-\bfe_{t'}}e_{\bfv^{\bfi}(\bfa+\bfe_t)} \gkk{t}{\bvi{\bfa}} \gkk{t'}{\bfv^{\bfi}(\bfa-\bfe_{t'})} \ep{l-1}{\bfv^{\bfi}(\bfa-\bfe_{t'})}).\\
}}
A path $\gkk{t}{\bvi{\bfa}} \gkk{t'}{\bfv^{\bfi}(\bfa-\bfe_{t'})} $ appeared in \eqref{hkereq} ends at vertex $\bvi{\bfa'}$ for $\bfa'=\bfa+\bfe_t \in \Zp^{n+1}_{l+1}$.

Assume that $\bfa'=(a_1',\ldots,a_{n+1}')\in \Zp^{n+1}_{l+1}$ and there is an arrow from $\bvi{\bfa}$ to $\bvi{\bfa'}$ for some $\bfa \in C\sk{n}_l(\bfi)$.
Thus $|T(\bfa')| \le 1$, and $\bfa'\in C\sk{n}_{l+1}$ if and only if   $|T(\bfa')| = 0$.

If $\bfa'\in C\sk{n}_{l+1}$, $l+1-l-|T(\bfa')|=1$, then by Lemma ~\ref{kercompositionfactor},  $\dmn_k  e_{\bvii(\bfa')}(\Ker f_l)_{l+1}=1$.

If  $ \bfa'-\bfe_t, \bfa'-\bfe_{t'}\in  C\sk{n}_{l}(\bfi)$ for some $t'\neq t$, then $0\le a'_{t} = (\bfa'-\bfe_{t'})_{t} \le b_t$,  $0\le a'_{t'} = (\bfa'-\bfe_{t})_{t'} \le b_{t'}$, and $0\le a_{t''}' = (\bfa'-\bfe_{t})_{t''} \le b_{t''}$ for $t'' \not\in \{t,t'\}$.
This implies that $\bfa'-\bfe_t-\bfe_{t'} \in C\sk{n}_{l-1}(\bfi)$. 
Left multiply  \eqref{hkereq} with $e_{\bfv^{\bfi}(\bfa')}$, we get,
\eqqc{hkereq2}{\arr{ccl}{  0&=& \sum\nolimits_{1\le t<t'\le n+1, \bfa'-\bfe_t-\bfe_{t'} \in C\sk{n}(\bfi)}
\\ &&
(\zks{\bfa'-\bfe_t,t}\zk{l-1}{\bfa'-\bfe_t-\bfe_{t'},t'} -
\zks{\bfa'-\bfe_{t'},t'}\zk{l-1}{\bfa'-\bfe_{t'}-\bfe_t,t}d_{t',t,\bfa'-\bfe_{t'}-\bfe_t} ) \\ &&
e_{\bfv^{\bfi}(\bfa')} \gkk{t}{\bfv^{\bfi}(\bfa'-\bfe_t)}
\gkk{t'}{\bfv^{\bfi}(\bfa'-\bfe_t-\bfe_{t'})} \ep{l-1}{\bfv^{\bfi}(\bfa'-\bfe_t-\bfe_{t'})}.\\
}}
So we get a system of linear equations
\eqqc{hcoeeq}{\arr{ccl}{
\zks{\bfa'-\bfe_t,t}\zk{l-1}{\bfa'-\bfe_t-\bfe_{t'},t'} -
\zks{\bfa'-\bfe_{t'},t'} \zk{l-1}{\bfa'-\bfe_{t'}-\bfe_t,t} d_{t',t,\bfa'-\bfe_{t'}-\bfe_t} &=&0
}}
for $1\le t<t'\le n+1$ with $\bfa'-\bfe_t-\bfe_{t'} \in C\sk{n}_{l-1}(\bfi)$.
The solution spaces of \eqref{hcoeeq} is  one dimensional, since $\dmn_k  e_{\bvii(\bfa')}(\Ker f_l)_{l+1}=1$.
Take a non-zero solution   $$\zks{\bfa'-\bfe_t,t} = \zk{l}{\bfa'-\bfe_t, t}\in k,\bfa'-\bfe_{t} \in C\sk{n}_l(\bfi)$$ 
for \eqref{hcoeeq}.
Then $\zk{l}{\bfa'-\bfe_t,t} \neq 0$ if $a'_t \neq 0$, and we take $\zk{l}{\bfa'-\bfe_t,t} =0$ for those $t$ with $a'_t = 0$, thus
\eqqc{hkele}{\ta{l}{\bvii(\bfa')} = \sum\limits_{t=1}^{n+1} \zk{l}{\bfa'-\bfe_t,t} \gkk{t}{\bvi{\bfa'-\bfe_t}} \ep{l}{\bvi{\bfa'-\bfe_t}}}
is in  $e_{\bfa'}(\Ker f_{l})_{l+1}$. 

Otherwise, there is $t$ such that $a'_t>0 $ and $a'_{t'}=0$ for $t'\neq t$, left multiply \eqref{hkereq} with $e_{\bfv^{\bfi}(\bfa')}$, we get a zero term:
\eqqc{hcoeeqsingle}{
\zks{\bfa'-\bfe_t,t}\zk{l-1}{\bfa'-\bfe_t-\bfe_{t},t}
e_{\bfv^{\bfi}(\bfa')} \gkk{t}{\bfv^{\bfi}(\bfa'-\bfe_t)}
\gkk{t}{\bfv^{\bfi}(\bfa'-\bfe_t-\bfe_{t})} \ep{l-1}{\bfv^{\bfi}(\bfa'-\bfe_t-\bfe_{t})}=0.
}
The solution spaces of \eqref{hcoeeqsingle} is one dimensional, since $\dmn_k  e_{\bvii(\bfa')}(\Ker f_l)_{l+1}=1$.
Take  $\zks{\bfa'-\bfe_t,t}=1$ for \eqref{hcoeeqsingle},
\eqqc{hkelesingle}{\ta{l}{\bvi{\bfa'}} =  \gkk{t}{\bvi{\bfa'-\bfe_t}} \ep{l}{\bvi{\bfa'-\bfe_t}},}
is a  non-zero element of $e_{\bfa'}(\Ker f_{l})_{l+1}$.

This shows that $K_{l}$ is a subset of $(\Ker f_{l})_{l+1}$.

If $\bfa' \not \in C\sk{n}_{l+1}$, then $|T(\bfa')| = 1$, then  $l+1-l-|T(\bfa')|=0$, so by Lemma ~\ref{kercompositionfactor},  $\dmn_k  e_{\bvii(\bfa')}(\Ker f_l)_{l+1}=0$.
Thus $\tS(\bvi{\bfa'})$ is not a composition factor of $(\Ker f_l)_{l+1}$.
This proves that the set $K_{l}$ spans the space $(\Ker f_l)_{l+1}$.

Now we prove that $K_{l}$ generates $\Ker f_l$.

Since $f_l$ is degree zero map, $\Ker f_l$ is graded.
Consider the degree $l+s$ component  of \eqref{hprojres}, for each $2\le s \le n+1$.
Then by Lemma ~\ref{kercompositionfactor}, for each $\bfa'\in \Zp_{l+s}^{n+1}$, we have
$$ \dmn_k e_{\bvi{\bfa'}}(\Ker f_l)_{l+s} = \cmb{|\mR(\bfa')|-1}{s-|T(\bfa')|-1}.$$

Now we consider the submodule generated by $K_l$.
For each subset $P$ of $s-|T(\bfa')|-1$ elements in $\mR(\bfa')$,  let $$\arr{rcl}{\kappa(P,\bfa) & =& \gkk{T(\bfa')}{\bvi{\bfa'-\bfe_{T(\bfa')}}} \gkk{P}{\bvi{\bfa}} \ta{l}{\bvi{\bfa}} = \gkk{T(\bfa')}{\bvi{\bfa'-\bfe_{T(\bfa')}}} \gkk{P}{\bvi{\bfa}} \sum\limits_{t=1}^{n+1} \zk{l}{\bfa-\bfe_t,t} \gkk{t}{\bvi{\bfa-\bfe_t}} \ep{l}{\bvi{\bfa-\bfe_t}}\\
& =& \sum\limits_{t\in \mR(\bfa') \setminus  P} \zk{l}{\bfa-\bfe_t,t} \gkk{T(\bfa')}{\bvi{\bfa'-\bfe_{O(T)}}} \gkk{P}{\bvi{\bfa}} \gkk{t}{\bvi{\bfa-\bfe_t}} \ep{l}{\bvi{\bfa-\bfe_t}},
}$$
for any $\bfa = \bfa'-\bfe_{T(\bfa') \cup P} \in C\sk{n}_{l+s-|P|-|T(\bfa')|} (\bfi)$.
Then $\kappa(P,\bfa) \in (\tL(n)K_l)_{l+s}$.

Let $s''=l+s-|P|-|T(\bfa')|$.
Fix a $t_0\in \mR(\bfa')$, and let $\mathcal P(s'', t_0) = \{P\subset \mR(\bfa') \setminus  \{t_0\})\mid  |P|=s''\}$.
Assume that $$\sum\limits_{P\in \mathcal P(s'', t_0),  \bfa\in C\sk{n}_{s''} (\bfi)} d_{P,\bfa}\kappa(P,\bfa)  = 0 $$ for $d_{P,\bfa}\in k$, that is
\small\eqqc{lind}{\arr{lll}{0&=& \sum\limits_{P\in \mathcal P(s'', t_0),  \bfa\in C\sk{n}_{s''}(\bfi)} d_{P,\bfa} \kappa(P,\bfa)  \\
& =& \sum\limits_{P\in \mathcal P(s'', t_0), \bfa\in C\sk{n}_{s''}(\bfi), t\in \mR(\bfa') \setminus P} d_{P,\bfa} \zk{l}{\bfa-\bfe_t,t} \gkk{T(\bfa')}{\bvi{\bfa'-\bfe_{T(\bfa')}}} \gkk{P}{\bvi{\bfa}} \gkk{t}{\bvi{\bfa-\bfe_t}} \ep{l}{\bvi{\bfa-\bfe_t}}\\
& =& \sum\limits_{\bfa\in C\sk{n}_{s''-1}(\bfi), P\in \mathcal P(s'', t_0),t\in \mR(\bfa'),t\not\in P} d_{P,\bfa+\bfe_t} \zk{l}{\bfa,t} \gkk{T(\bfa')}{\bvi{\bfa'-\bfe_{T(\bfa)}}} \gkk{P}{\bvi{\bfa+\bfe_t}} \gkk{t}{\bvi{\bfa}} \ep{l}{\bvi{\bfa}}\\
& =& \sum\limits_{\bfa\in C\sk{n}_{s''-1}(\bfi), t,t'\in \mR(\bfa')\setminus \{t_0\}, P\subset \mR(\bfa')\setminus\{t_0\}, |P| =s''-1}\\ &&
\qquad (d_{P\cup\{t'\},\bfa+\bfe_t}h_{t', P}\zk{l}{\bfa,t} +d_{P\cup\{t\},\bfa+\bfe_{t'}}h_{t, P}c(t,t')\zk{l}{\bfa,t'}) \\ && \qquad \gkk{T(\bfa')}{\bvi{\bfa'-\bfe_{T(\bfa)}}} \gkk{P}{\bvi{\bfa+\bfe_{t,t'}}}\gkk{t'}{\bvi{\bfa+\bfe_t}} \gkk{t}{\bvi{\bfa}} \ep{l}{\bvi{\bfa}}\\
& & + \sum\limits_{\bfa\in C\sk{n}_{s''-1}(\bfi), P\subset \mathcal P(s'',t_0)}
d_{P,\bfa+\bfe_{t_0}}\zk{l}{\bfa,t_0} \gkk{T(\bfa')}{\bvi{\bfa'-\bfe_{T(\bfa)}}} \gkk{P}{\bvi{\bfa+\bfe_{t_0}}}\gkk{t_0}{\bvi{\bfa}}\ep{l}{\bvi{\bfa}},
}}\normalsize
for some $h_{t', P}, h_{t, P}, c(t,t') \in k$.
Thus $d_{P,\bfa}=0$ for all $P\in \mathcal P(s'', t_0) $ and $\bfa \in C\sk{n}_{s''}(\bfi)$ with $\bfa -\bfe_{t_0} \in C\sk{n}_{s''-1}(\bfi)$.

For  $P\in \mathcal P(s'', t_0) $ and $\bfa \in C\sk{n}_{s''}(\bfi)$, we have  $e_{\bvi{\bfa'}} \gkk{T(\bfa')}{\bvi{\bfa'-\bfe_{T(\bfa)}}} \gkk{P}{\bvi{\bfa}} \ep{l}{\bvi{\bfa}} \neq 0$, then $\bfa = \bfa' - \bfe_{T(\bfa) \cup P}$ and $a_{t_0}= a'_{t_0} $.
Thus $a'_{t_0}>0$, $a_{t_0} >0$ and  $\bfa -\bfe_{t_0} \in C\sk{n}_{s''-1}(\bfi)$. 
This proves that $d_{P,\bfa}=0$ for all $P\in \mathcal P(s'', t_0)$ and $ \bfa\in C\sk{n}_{s''} (\bfi)$.

Thus $\{\kappa(P,\bfa)\mid  P\in \mathcal P(s'', t_0),  \bfa\in C\sk{n}_{s''}(\bfi)\} $ is a linearly independent set and
$$\dmn_k (\tL(n) K_l)_{l+s } \ge \cmb{|\mR(\bfa')|-1}{s-|T(\bfa')|-1}.$$
Since $(\tL(n) K_l)_{l+s} \subset (\Ker f_l)_{l+s}$, this implies that  $(\tL(n) K_l)_{l+s} =(\Ker f_l)_{l+s}$ for $s=2,\ldots,n+1$.

This proves that $\Ker f_l$ is generated by $K_l$ for $0\le l \le m-2$, and so we prove (a), (b) and (c), by induction.

Now we prove (d).
Consider the case of $l=m-1$, note that $C\sk{n}_{m-1}(\bfi) = \{\bfb =(b_1,\ldots,b_{n+1})\}$, with $b_1= m+n-1-\sum\limits\limits_{t=1}^{n} i_t$, and $b_t=i_{t-1}-1$ for $2\leq t\leq n+1$.

Clearly, we have $\soc \tP(\bvi{\bfb}) =\tP(\bvi{\bfb})_{n+1}=  (\Ker f_{m-1})_{n+1}\subset \Ker f_{m-1} $.
For each $\bfa'\neq \bfb+\bfe$ such that $\tS(\bvi{\bfa'})$ is a composition factor of $\bfr \tP(\bvi{\bfb})_s$, for $1\le s \le n$, $e_{\bvi{\bfa'}}\tL e_{\bvi{\bfb}} \neq 0$ and $\bfa'\in \Zp^{n+1}_{m-1+s} \setminus C\sk{n}(\bfi)$.
Considering  the composition factor $\tS(\bfa')$ in the degree $m-1+s$ component of the projective resolution \eqref{hprojres}.
Since the only vertex of the form $\bvi{\bfa}$ with $\bfa \in C\sk{n}(\bfi)$ on a path from $\bvi{\bfb}$ to $\bvi{\bfa'}$ is $\bvi{\bfb}$,  $|T(\bfa)|= s$.
So $s-|T(\bfa)|=0$, and by Lemma ~\ref{kercompositionfactor}, $\dmn_k  e_{\bvi{\bfa'}}(\Ker f_{m-1})_{m-1+s} = 0. $
This proves that $(\Ker f_{m-1})_{m-1+s} =0$ for $1\le s \le n$, and thus $\Ker f_{m-1} = \soc \tP(\bvi{\bfb})$.
\end{proof}

Let $\tL_0$ be the subspace spanned by the idempotents of $\tL(n)$, and let $\tL_1$ be the subspace spanned by the arrows.
Then $\tilde{\rho}^{\bbfg}(n) \subset \tL_1\otimes_{\tL_0} \tL_1$, and it spans a subspace $R^{\bbfg} \subset \tL_1\otimes_{\tL_0} \tL_1$.
Let $T_{\tL_0}(\tL_1) = \tL_0  + \tL_1 + \tL_1\otimes_{\tL_0} \tL_1 + \cdots \tL_1^{\otimes_{\tL_0} t}+\cdots$ be the tensor algebra and let $(R^{\bbfg})$ be the ideal generated by $R^{\bbfg}$, then $\tL^{\bbfg}(n) \simeq T_{\tL_0}(\tL_1) /(R^{\bbfg})$ is a quadratic algebra.
The quadratic dual $\tL^!(n)$ of $\tL(n)$ is defined as the quotient $ T_{\tL_0}(D\tL_1) /(R^{\bbfg,\perp})$, where $D\tL_1$ is the dual space of $\LL_1$ and $R^{\bbfg, \perp}$ is the annihilator of $R^{\bbfg}$ in $D(\tL_1\otimes_{\tL_0} \tL_1)$.
From a different view point, Theorem ~\ref{trata} is restated as following theorem:

\begin{theorem}\label{ncAkszul}
Assume that $m \ge 3$.

$\tL^{\bbfg}(n)$ is an almost Koszul algebra of type $(n+1, m-1)$, and its quadratical dual $\tL^{\bbfg,!}(n)$ is almost Koszul algebra of type $(m-1, n+1)$.

Both $\tL^{\bbfg}(n)$ and $\tL^{\bbfg,!}(n)$ are  periodic algebras.

More precisely, the minimal periodicity of $\tL^{\bbfg}(n)$ is $(n+1)m$,  and the minimal periodicity of $\tL^{\bbfg,!}(n)$ is $(n+1)(n+2)$.
\end{theorem}
\begin{proof}
If follows directly from Lemma ~\ref{hproofof} that $\tL^{\bbfg}(n)$ is an almost Koszul algebra of type $(n+1, m-1)$.
So $\tL^{\bbfg,!}(n)$ is almost Koszul algebra of type $(m-1, n+1)$, by Proposition 3.11 of \cite{bbk}.

Note that for each $\bfi \in \tQ(n)_0$, by (d) of Lemma ~\ref{hproofof}, we have that $\Omega^{m} \tS(\bfi) \simeq \tS(\bvi{\bfb(\bfi)}) = \tS(\omega(\bfi))$.
If $\bfi =(i_1,\ldots,i_n)$, then $$\bfb(\bfi) =(m+n-1-\sum\limits_{t=1}^{n} i_t, i_1-1, i_2-1, \ldots, i_{n}-1) $$
and we have
$$\omega(\bfi) = \bvi{\bfb(\bfi)} = (m+n-\sum\limits_{t=1}^{n} i_t, i_1, \ldots, i_{n-1}).$$
Thus $\omega^t(\bfi) = (i_{n-t+2},\ldots, i_{n}, m+n-\sum\limits_{t=1}^{n} i_t, i_1, \ldots, i_{n-t}),$ and $\omega^{n+1}(\bfi)= \bfi$.
This proves that $\Omega^{m(n+1)} \tS(\bfi) =\tS(\bfi)$, and $n+1$ is minimal such that this holds for all $\bfi \in \tQ(n)_0$.
So the minimal periodicity of $\tL^{\bbfg}(n)$ is $(n+1)m$.

Let $N(\tL^{\bbfg}(n),m-1)$ be the matrix with $$N(\tL^{\bbfg}(n),m-1)_{\bfi\bfj} =\dmn_k e_{\bfi} \Omega^{m}\tL^{\bbfg}_0(n)e_{\bfj}.$$
$N(\tL^{\bbfg}(n),m-1)$ is the matrix of the permutation of the simples $\tL^{\bbfg}(n)$-modules defined by $\Omega^{m}$, so the order of $N(\tL^{\bbfg}(n),m-1)$ is $n+1$.

Now let $N(\tL^{\bbfg,!}(n),n+1)$ be the matrix with $$N(\tL^{\bbfg,!}(n),n+1)_{\bfi\bfj} =\dmn_k e_{\bfi} \Omega^{n+2}\tL^{\bbfg,!}_0(n)e_{\bfj}.$$
By Proposition 3.14 of \cite{bbk}, $N(\tL^{\bbfg,!}(n),n+1)$ is the matrix of the permutation of the simples $\tL^{\bbfg,!}(n)$-modules defined by $\Omega^{n+2}$ and it is the transpose of  $N(\tL^{\bbfg}(n),m-1)$.
So they have the same order $n+1$, and this proves that the minimal periodicity of $\tL^{\bbfg,!}(n)$ is $(n+1)(n+2)$.
\end{proof}

\section{$(n+1)$-cuboid Truncations}

Assume that $\LL(n)=\LL^{\bfg}(n)$ is an $n$-translation algebra with admissible $n$-translation quiver $Q(n)$ in this section.
Let $\bbfg=(\bfg,g')$ be  as  defined in Section 2, let $\tL(n)=\tL^{\bbfg}(n)$ be a twisted trivial extension of $\LL(n)$.
Now consider truncations on $\olL(n) = \olL^{\bbfg}(n)=\tL^{\bbfg}(n)\# \mathbb Z^*$ and we write $\olrho(n)$ for $\olrho^{\bbfg}(n)$.
By Theorem ~\ref{ncAkszul} and Proposition 5.4 of \cite{gu14}, $\olL(n)$ is an $n$-translation algebra with stable  $n$-translation quiver $\olQ(n)$ and $n$-translation $\otau_{n} \bfi = \bfi-\bfe_{n+1}$.
Write $\otau = \otau_{n}$ for the $n$-translation of $\olQ(n)$.

Let $\olS(\bfi)$ be the simple $\olL(n)$-module at $\bfi\in \olQ(n)_0$ and let $\olP(\bfi)$ and $\olI(\bfi)$ be its projective cover and injective envelop as $\olL(n)$-modules, respectively.
Conventionally,  set $\olS(\bfi)$, $\olP(\bfi)$ and $\olI(\bfi)$ to be zero when $\bfi \not \in \olQ(n)_0$.

For each $\bfi =(i_1,\ldots,i_n,i_{n+1}) \in \olQ(n)_0$, let $\bfi' = (i_1,\ldots,i_n)$ be its $n$-truncation in $\tQ(n)_0 = Q(n)_0$.
Define its $(n+1)$-cuboid $C\sk{n}(\bfi) = C\sk{n}(\bfi')$.
We have characterization of the linear part of the  projective of simple $\olL^{\bbfg}(n)$-modules using the $(n+1)$-cuboids $C\sk{n}(\bfi)$ similar to Proposition ~\ref{pprojresta}.
Let \eqqc{hprojresba}{\cdots \lrw \olP^2(\olS(\bfi))\stackrel{f_2}{\lrw} \olP^1(\olS(\bfi))\stackrel{f_1}{\lrw} \olP^0(\olS(\bfi))\stackrel{f_0}{\lrw}  \olS(\bfi) \lrw 0 }
be a minimal projective resolution of the simple $\olL$-modules $\olS(\bfi)$ of corresponding to the vertex $\bfi\in \olQ(n)_0$.
If $\bfi\in \olQ_0$, call $\bfa$ an {\em $\bfi$-quiver vertex} if $\bbvi{\bfa} \in \olQ(n)_0$.

Consider the linear part of this projective resolution.
We have
\begin{proposition}\label{pprojresba} For $l=0, 1, \ldots, m-1$ and $\bfi\in \olQ(n)_0$, we have that $\olP^l(\olS(\bfi))$ is generated in degree $l$ and
\eqqc{bhprtml}{\olP^l(\olS(\bfi)) \simeq \bigoplus_{\bfa\in C\sk{n}_l({\bfi})} \olP(\bbvi{\bfa}).}

Further more,  $C\sk{n}_{m-1}({\bfi}) =\{\bfb\}$, and $\Ker f_{m-1} \simeq \soc \olP(\bbvi{\bfb}) =\olS(\bbvi{\bfb + \bfe})$ is simple.
\end{proposition}

Recall that a full bound subquiver $Q'$  of $\olQ(n)$ is called a  {\em $\otau_n$-slice} of $Q$ if it has the following property \cite{g12}:
\begin{enumerate}
\item[(a)] for each vertex $\bfi $ of $\olQ(n)$, the intersection of the $\otau_{n}$-orbit of $v$ and the vertex set of $Q'$ is a single-point set;
\end{enumerate}
A $\otau_{n}$-slice is called a {\em complete $\otau_{n}$-slice}, if it also has the following property:
\begin{enumerate}
\item[(b)] $Q'$ is path complete in the sense that for any path $p: v_0 \lrw v_1 \lrw \ldots \lrw v_t$ of $\olQ(n)$ with $v_0$ and $v_t$ in $Q'$, the whole path $p$ lies in $Q'$.
\end{enumerate}
The algebra defined by a complete $\otau$-slice of $\olQ(n)$ is called a {\em $\otau$-slice algebra} of $\olQ(n)$ (or of $\olL(n)$).

Let $Q(n)\times \{t\}$ be the full bound subquiver of $\olQ(n)$ with vertex set
$Q(n)_0\times \{t\}$ for any $t\in \mathbb Z$.
It is easy to check that we have the following result.

\begin{proposition}\label{tauslice} $Q(n)\times \{t\}$ are complete $\otau$-slices of $\olQ(n)$ for all $t$.
\end{proposition}
Let $\LL(n,t)$ be the algebra defined by the bound quiver $Q(n)\times \{t\}$.
$\LL(n,t)$ are isomorphic to $\LL^{\bfg} (n)$ for all $t$.

Now we define the {\em $(n+1)$-cuboid completion} $Q(n+1)(t)$ of $Q(n)\times \{t\}$ as the full subquiver of $\olQ(n)$ with the vertex set:
$$Q(n+1)(t)_0 = \cup_{\bfi \in Q(n)_0\times \{t\}}\{\bbvi{\bfa}|\bfa\in C\sk{n}(\bfi)\} .$$
Clearly, $Q(n+1)(t)$ are isomorphic  quivers for all $t$.

Let $Q(n+1)(1)\sk{t}$ and $ Q(n+1)\sk{t}$ be the full subquivers of  $Q(n+1)(1)$ and $Q(n+1)$ with vertex set $Q(n+1)(1)\sk{t}_0 =\{ \bfi = ( i_1, \ldots, i_t, 1, \ldots,1) \in Q(n+1)(1)_0\}$ and  $Q(n+1)\sk{t}_0 =\{ \bfi = (i_1, \ldots, i_t, 1, \ldots, 1) \in Q(n+1)_0\}$, respectively.

\begin{lemma}\label{boundsubq}
For $1\le t \le n+1$, $(Q(t),\rho(t))$ can be identified with $$(Q(n+1)(1)\sk{t}, \olrho(n)[Q(n+1)(1)\sk{t}])$$ as full bound subquivers.
\end{lemma}
\begin{proof}
We prove by using induction on $n$ and $t$.
The assertion clearly holds for $n=0$ since  $Q(1)(1)\sk{1} = Q(1)$.

Assume that $n>0$ and the lemma holds for $n' < n$.
By inductive assumption $(Q(t),\rho(t))$ can be identify $Q(n)(1)\sk{t}$, and $Q(n)(1)=Q(n)(1)\sk{n}$, which is identified with $Q(n)$ and with  $Q(n)\times \{1\} = Q(n+1)(1)\sk{n}$, a full subquiver of $Q(n+1)$ inside $Q(n+1)(1)$, if $t \le n$.

We need only to prove the case of $t=n+1$.

For any $\bfi = (i_1,\ldots,i_n,i_{n+1})\in Q(n+1)_0$, we have that $(i_1,\ldots,i_n) \in Q(n)_0$ and $i_{n+1}\le m+n-1-\sum_{t=1}^n i_t$, by definition.
It follows easily that $i_n+i_{n+1}-1 \le m+n-1-\sum_{t=1}^{n-1} i_t$, thus $\bfi'=(i_1,\ldots, i_{n-1}, i'_n=i_n+i_{n+1}-1,1) \in Q(n)_0\times \{1\}$.
Let $\bfa=(0,\ldots,0, i_{n+1}-1)$, then $\bfa\in C\sk{n}(\bfi')$,  and we have that
$$\bar{\bfv}^{\bfi'}(\bfa) = (i_1,\ldots,i_n,i_{n+1}) =\bfi.$$
This proves that $Q(n+1)_0\subset Q(n+1)(1)_0$.

On the other hand, for any $\bfi = (i_1, \ldots, i_n, i_{n+1}) \in Q(n+1)(1)$, there is an $\bfi'= (i_1',\ldots,i_n',1) \in Q(n+1)_0 \times\{1\}$ such that $\bfi=\bar{\bfv}^{\bfi'}(\bfa)$ for some $\bfa\in C\sk{n}(\bfi')$.
Thus $i_t = i_t'+a_t-a_{t+1}>1$ for $t\le n$ and $i_{n+1}=1+a_{n+1}$.
For $1\le s \le n$, we have
$$\arr{rcl}{\sum\limits_{t=1}^s i_t &=&\sum_{t=1}^s i'_t +a_1-a_{s+1} \\ &\le &  \sum_{t=1}^s i'_t + m+n-1-\sum\limits_{t=1}^{n} i'_t -a_{s+1}\\
&=& m+n-1-\sum\limits_{t=s+1}^{n} i'_t -a_{s+1} \le m+s-1, }$$
and $$\sum\limits_{t=1}^{n+1} i_t =\sum_{t=1}^{n+1} i'_t +a_1\le m+n.$$
This proves that $\bfi \in Q(n+1)_0$ and thus $Q(n+1)_0= Q(n+1)(1)_0$.

Now we prove that $Q(n+1)(1)$ is a full bound subquiver of $\olQ(n)$.
Assume that $\bfj,\bfj(t),\bfj(s,t) \in Q(n+1)(1)_0$ and $\bfj(s) \in \olQ(n)_0$.
Then $\bfi = \bfj-(j_{n+1}-1)\bfe_{n+1} \in Q(n+1)(1)\sk{n}_0$ and $\bfi(s)=\bfj(s)-(j_{n+1}-1)\bfe_{n+1}, \bfi(t)=\bfj(t)-(j_{n+1}-1)\bfe_{n+1}, \bfi(s,t)=\bfj(s,t)-(j_{n+1}-1)\bfe_{n+1} \in Q(n+1)(1)\sk{n}_0 $.
Thus $\bfe_t+(j_{n+1}-1)\bfe_{n+1}, \bfe_t+\bfe_s +(j_{n+1}-1)\bfe_{n+1} \in C\sk{n}(\bfi)$, so $\bfe_s +(j_{n+1}-1)\bfe_{n+1} \in C\sk{n}(\bfi)$ and $\bfj(s)=\bvii({\bfe_s +(j_{n+1}-1)\bfe_{n+1} }) \in Q(n+1)(1)_0$.
This proves that $ (Q(n+1)(1), \olrho(n)[Q(n+1)(1)])$ is a full bound quiver of $(\olQ(n), \olrho(n))$, since elements of $\olrho(n)$ are either paths of length $2$, or linear combinations of two paths of length $2$ of the arrows of the same types.
\end{proof}

Write $\rho(n+1)(1)\sk{t}=\olrho(n)[Q(n+1)(1)\sk{t}]$, and $\rho(n+1)(t) = \olrho(n)[Q(n+1)(t)]$.
As a corollary of Lemma ~\ref{boundsubq}, we have
\begin{corollary}\label{cboundsubq}
$(Q(n+1)(t),\rho(n+1)(t))$ are full bound subquivers of $(\olQ(n), \olrho(n))$.
\end{corollary}

Let $\LL(n+1)(t)$ be the algebra defined by the bound quiver $(Q(n+1)(t), \rho(n+1)(t))$ and they will be called   {\em $(n+1)$-cuboid truncations} of $\olL^{\bbfg}(n)$.

We have the following theorem.

\begin{theorem}\label{cuboidtruncation}
\begin{itemize}
\item[(a)] $\LL(n+1)(t)$  are all isomorphic to $\LL(n+1)(1)=\LL(n+1)$.

\item[(b)] Let $I(t)$ be the ideal of $\olL(n)$ generated by the set $\{e_{\bfi} \mid \bfi \in \olQ(n)_0\setminus Q(n+1)(t)_0 \}$, then $\LL(n+1) \simeq \olQ(n)/I(t)$ for all $t$.

\item[(c)] $\LL(n+1)$ is $n$-translation algebra.
\end{itemize}
\end{theorem}

\begin{proof}
(a), (b) follows directly from Lemma ~\ref{boundsubq}.

Now we prove (c).
Identify $Q(n+1)$ with $Q(n+1)(1)$.
As a truncation of an $n$-translation quiver, it follows from (a) and (b) that $Q(n+1)$ is an $n$-translation quiver.

For each $\bfi \in Q(n+1)_0 = Q(n+1)(1)_0 \subset \olQ(n)_0$, we have a minimal projective resolution \eqref{hprojresba} of simple $\olL(n)$-module $\olS(\bfi)$.
Note that $\LL(n+1) \simeq \olL(n) / I$, where $I$ is the ideal generated by $\{ e_{\bfj}\mid  \bfj \in \olQ(n)_0\setminus Q(n+1)(1)_0\} $.
Since $Q(n+1)(1)_0$ is finite let $e=\sum\nolimits_{\bfi \in Q(n+1)(1)_0} e_{\bfi} $, then $\bar{e}$ is the unit of $\olL(n)/I$.
We have $S(\bfi) =\olS(\bfi)$ for $\bfi \in Q(n+1)(1)_0$.
Tensor \eqref{hprojresba} with $\olL(n)/I$, we get a complexes of $\LL(n+1)$-modules
\eqqc{projresnpo}{\arr{l}{\cdots \lrw \olL(n)/I\otimes_{\olL(n)}\olP^2(\olS(\bfi))\stackrel{1\otimes f_2}{\lrw} \olL(n)/I\otimes_{\olL(n)} \olP^1(\olS(\bfi)) \\ \qquad \stackrel{1\otimes f_1}{\lrw} \olL(n)/I\otimes_{\olL(n)} \olP^0(\olS(\bfi))\stackrel{1\otimes f_0}{\lrw}  \olL(n)/I\otimes_{\olL(n)}\olS(\bfi) = S(\bfi) \lrw 0 .}}
Note that $\olL(n)/I\otimes_{\olL(n)} \olP^t(\olS(\bfi)) \simeq \olP^t(\olS(\bfi))/ I\olP^t(\olS(\bfi))$.
For any $y=\sum_p \bar{a}_p \otimes y_p =\bar{e }\otimes \sum_p a_p y_p $ in $\Ker (1\otimes f_t)$, we have that $\sum_p a_p y_p \in \Ker f_t = f_{t+1}(\olP^{t+1}(\olS(\bfi)))$ and hence there is an $x \in \olP^{t+1}(\olS(\bfi))$ such that $f_{t+1}(x)= \sum_p a_p y_p $.
Thus $(1\otimes f_{t+1})(\bar{e } \otimes x) = y$.
This proves that \eqref{projresnpo} is exact and hence a projective resolution $S(\bfi)$ as $\LL(n+1)$-modules.

Note that $\olL(n)/I\otimes_{\olL(n)} \olP^t(\olS(\bfi)) \simeq \olP^t(\olS(\bfi))/ I\olP^t(\olS(\bfi))$.
Thus if $\bfi \in \olQ(n)_0\setminus Q(n+1)(1)_0$, then $\olL(n) / I \otimes_{\olL(n)} \olL(n) e_{\bfi} = (\olL(n)/I \olL(n)) e_{\bfi} \otimes_{\olL(n)} e_{\bfi} =0$ and if $\bfi \in Q(n+1)(1)_0$, then $\olL(n)/I\otimes_{\olL(n)} \olL(n)e_{\bfi} =(\olL(n)/I \olL(n) e_{\bfi} \simeq \LL(n+1) e_{\bfi}.$
If $\bfi\in Q(n+1)(1)_0$, by Proposition ~\ref{pprojresba}, we have for $l=0, 1, \ldots, m-1$,
\eqqc{bhprtmlnpo}{\olL(n)/I\otimes_{\olL(n)} \olP^l(\olS(\bfi)) \simeq \bigoplus\limits_{\bfa\in C\sk{n}_l({\bfi}), \bvii{\bfa}\in  Q(n+1)(1)_0} \LL(n+1) (1) e_{\bbvi{\bfa}}}
is generated in degree $l$ and $\Ker (1\otimes f_{m-1})$ is either simple or zero.

This proves that $\LL(n+1)$ is almost Koszul of type $(n+1,m-1)$ and thus is $n$-translation algebra.

\end{proof}

We call $\LL(n+1)$ a {\em $(n+1)$-cuboid completion of $\LL(n)$}.

\section{Admissibility of Pyramid $n$-cubic Algebra}

Now fix $m \ge 3$, and a integer $n\ge 1$.
Let $Q(n)$ and $\rho(n)=\rho^{\bfg}(n)$  be the data for a pyramid shaped $n$-cubic quiver $Q(n) = Q^{\bfg}(n)$ as defined in \eqref{thequiver}, \eqref{elementg} and \eqref{therelation}, and let $\Lambda(n)$ be the pyramid $n$-cubic algebra defined by the bound quiver $(Q(n),\rho(n))$.

For a bound path $p$ from $i$ to $j$ such that any of bound path from $i$ to $j$ of the same length as $p$ is linearly dependent on $p$, recall that $p$ is {\em left stark of degree $t$} with respect to $i'$ if $pw$ is a bound element for any bound element $w$ from $i'$ to $i$ of length $t< n+1-l(p)$; and $p$ is {\em right stark of degree $t$} with respect to $j'$ if $wp$ is a bound element for any bound element $w'$ from $j$ to $j'$ of length $t< n+1-l(p)$.
A bound path $p$ is {\em shiftable} if it is linearly dependent to paths of the form $p'p''$ with $p''$ is trivial or linearly dependent on paths passing through no injective vertex and $p'$ is trivial or linearly dependent on paths passing through no projective vertex.

Recall that an $n$-translation quiver $Q$ is called {\em  admissible} if it satisfies the following conditions:
\begin{enumerate}
\item[(i)] For each bound path $p$, there are paths $q'$ and $q''$ such that $q'pq''$ is a bound path of length $n+1$.
\item[(ii)] Any bound path $p$ from a non-injective vertex $i$ to a non-projective vertex $j$ is linearly dependent to shiftable paths.
\item[(iii)] Let $i$ be a non-projective vertex.
     Let $p$ be a bound path ending at $i$ and let $q$ be a bound path starting at $\tau i$ with $l(p)+l(q) \le n$.
     If $p$ passes a projective vertex and $q$ passes an injective vertex, then $p$ is either left stark with respect to $t(q)$, or $q$ is right stark with respect to $s(p)$, of length $n+1- (l(p)+l(q))$.
\end{enumerate}

We have the following lemma describing the bound paths in $Q(n)$

\begin{lemma}\label{commpath}
Let $q = \gkk{t_r}{\bfi(t_1,\ldots,t_{r-1})} \gkk{t_{r-1}}{\bfi(t_1, \ldots, t_{r-2})} \cdots \gkk{t_1}{\bfi}$ be a bound path in $Q(n)$.

If $t_r\neq 1$ and $i_{t_r-1} >1$, then there is a $d\in k$ such that
$$q=d \gkk{t_{r-1}}{\bfi(t_1,\ldots,t_{r-2},t_r)} \gkk{t_{r-2}}{\bfi(t_1,\ldots,t_{r-3},t_r)} \cdots \gkk{t_1}{\bfi(t_r)} \gkk{t_r}{\bfi}.
$$

If $j_{t_1}>1$, then there is a $d\in k$ such that
$$q=d\gkk{t_{1}}{\bfi(t_2,\ldots,t_{r-1},t_r)} \gkk{t_{r}}{\bfi(t_2,\ldots,t_{r-1})} \cdots \gkk{t_3}{\bfi(t_2)} \gkk{t_2}{\bfi}.
$$
\end{lemma}
\begin{proof}
Set $\bfj\sk{u}=\bfi(t_1,\dots,t_{u-1})= (j_1\sk{u},\ldots,j_{n+1}\sk{u})$ for $u=1, \ldots, r$.
Since $q$ is a bound path, $ t_1  \neq t_r$ are pairwise different.

If $t_r\neq 1$ and $i_{t_r-1} >1$, then for $u=1, \ldots, r-1$, $j_{t_r-1}\sk{u}\ge i_{t_r-1} >1$, and we have that $\bfj\sk{u}(t_r) = \bfj^{u}-\bfe_{t_r-1} +\bfe_{t_r}$ is a vertex in $Q(n)$, too.
Thus there are non-zero $d_1,\ldots, d_{r-1} \in k$ such that $\gkk{t_r}{\bfi(t_1,\ldots,t_s)} \gkk{t_s}{\bfi(t_1,\ldots,t_{s-1})} = d_s \gkk{t_s}{\bfi(t_1,\ldots,t_{s-1},t_r)} \gkk{t_r}{\bfi(t_1,\ldots,t_{s-1})}$ and thus
$$\arr{lll}{q& = & \gkk{t_r}{\bfi(t_1,\ldots,t_{r-1})} \gkk{t_{r-1}}{\bfi(t_1,\ldots,t_{r-2})} \cdots \gkk{t_2}{\bfi(t_1)} \gkk{t_1}{\bfi}\\
& = &d_{r-1}\gkk{t_{r-1}}{\bfi(t_1,\ldots,t_{r-2},t_r)} \gkk{t_r}{\bfi(t_1,\ldots,t_{r-2})} \cdots \gkk{t_2}{\bfi(t_1)} \gkk{t_1}{\bfi}\\
&=& \cdots \\
&=&d_{r-1}\cdots d_1 \gkk{t_{r-1}}{\bfi(t_1,\ldots,t_{r-2},t_r)} \gkk{t_{r-2}}{\bfi(t_1,\ldots,t_{r-3},t_r)} \cdots \gkk{t_1}{\bfi(t_r)} \gkk{t_r}{\bfi} \\
&=&d \gkk{t_{r-1}}{\bfi(t_1,\ldots,t_{r-2},t_r)} \gkk{t_{r-2}}{\bfi(t_1,\ldots,t_{r-3},t_r)} \cdots \gkk{t_1}{\bfi(t_r)} \gkk{t_r}{\bfi}.
}$$
Similarly, if  ${t_1}>1$, then we have $j_{t_1-1}\sk{u}>1$ for $u= 2, \ldots, n$ so $(t_1)\bfj\sk{u}=\bfj\sk{u} + \bfe_{t_r-1} -\bfe_{t_r}\in Q(n)_0$.
This implies $(t_1)\bfj\sk{u}\in Q(n+1)_0$ and there are non-zero $d'_2, \ldots, d'_r \in k$ such that $\gkk{t_s}{\bfi(t_1,\ldots,t_{s-1})} \gkk{t_1}{\bfi(t_2,\ldots,t_{s-1})} = d'_s \gkk{t_1}{\bfi(t_2,\ldots,t_{s})} \gkk{t_s}{\bfi(t_2,\ldots,t_{s-1})}$ and thus
$$\arr{lll}{q& = & \gkk{t_r}{\bfi(t_1,\ldots,t_{r-1})} \gkk{t_{r-1}}{\bfi(t_1,\ldots,t_{r-2})} \cdots \gkk{t_2}{\bfi(t_1)} \gkk{t_1}{\bfi}\\
& = &d'_2 \gkk{t_r}{\bfi(t_1,\ldots,t_{r-1})} \gkk{t_{r-1}}{\bfi(t_2,\ldots,t_{r-2})} \cdots \gkk{t_1}{\bfi(t_2)} \gkk{t_2}{\bfi}\\
&=& \cdots \\
&=&d'_r d'_{r-1}\cdots d'_2 \gkk{t_{1}}{\bfi(t_2,\ldots,t_{r})} \gkk{t_{r}}{\bfi(t_2,\ldots,t_{r-1})} \cdots \gkk{t_3}{\bfi(t_2)} \gkk{t_2}{\bfi}
\\
&=&d' \gkk{t_{1}}{\bfi(t_2,\ldots,t_{r})} \gkk{t_{r}}{\bfi(t_2,\ldots,t_{r-1})} \cdots \gkk{t_3}{\bfi(t_2)} \gkk{t_2}{\bfi}.
}$$
\end{proof}

\begin{theorem}\label{main}
Pyramid $n$-cubic algebra $\LL(n)$ is an extendible $(n-1)$-translation algebra with admissible $(n-1)$-translation quiver $Q(n)$.
\end{theorem}
\begin{proof}
We prove the theorem by using induction on $n$.
We have
\eqqc{quiver0}{Q(1)=Q:\stackrel{1}{\circ}\longrightarrow \stackrel{2}{\circ} \longrightarrow \stackrel{}{\circ}\cdots \stackrel{}{\circ}\longrightarrow \stackrel{m}{\circ}, }
with the relation $\rho(1) = \{\alpha_{i+1} \alpha_i \mid  i = 1, \ldots, m-2\}$.
Clearly, $Q(1) = (Q(1),\rho(1))$ is an admissible $0$-translation quiver with $0$-translation $\tau_0:i+1\to i$ for $i=1,\ldots, m-1$.
$\Lambda(1)$ is a Koszul algebra with radical squared zero.
So it is a $0$-translation algebra, and it is extendible by Theorem 2.1 of \cite{bbk}.

Assume that $\LL(n)$ is an extendible $(n-1)$-translation algebra and its bound quiver $Q(n)$ is an admissible $(n-1)$-translation quiver.

It follows from Theorem ~\ref{cuboidtruncation} that $\LL(n+1)$ is an  $n$-translation algebra with $n$-translation $Q(n+1)$ and $n$-translation $\tau_n: \bfi \to \bfi - \bfe_{n+1}$.
We prove that $Q(n+1)$ is admissible.

Let $\cP$ and $\cI$ be the sets of projective vertices and injective vertices of $Q(n+1)$, respectively.
Then $$\mathcal P=\{\bfi\in Q(n+1)_0\mid  i_{n+1}=1\}\qquad \mathcal I=\{\bfi \in Q(n+1)_0 \mid \,  |\bfi| = m+n\}.$$
Let $p =\gkk{t_s}{\bfi(t_1,\ldots,t_{s-1})}\cdots \gkk{t_1}{\bfi}$ be a bound path in $Q(n+1)$.
If $t_h =1$, then $$\bfi, \bfi(t_1), \ldots, \bfi(t_1, \ldots, t_{h-1} ) $$ are not injective.
If $t_h =n+1$, then $$\bfi(t_1,\ldots,t_{h}), \bfi(t_1\ldots,t_{h}, t_{h+1}), \ldots, \bfi(t_1,\ldots,t_{s})$$ are not projective.

Assume that $p$ is a bound path from $\bfi$ to $\bfj$ in $Q(n+1)$.
Then $\bfj \in \{\bvi{\bfa}\mid \bfa \in U^{n+1}\}$ and $\bfi \in \{\bfv^{\bfj-\bfe_{n+1}}(\bfa)\mid \bfa \in U^{n+1}\}$.
If $\bfj$ is non-projective, then there is a path $p'$ from  $\tau_{n} \bfj$ to $\bfi$ such that $pp'$ is a bound path of length $n+1$, take $p''=e_{\bfj}$ be the trivial path.
If $\bfi$ is non-injective then there is a path $p''$ from $\bfj$ to $\tau_{n}^{-1} \bfi$  such that $p''p$ is a bound path of length $n+1$, take $p'=e_{\bfi}$ be the trivial path.
If $\bfi$ is injective and $\bfj$ is projective, then we have that $\sum\limits_{t=1}^{n+1} i_t = m+n$ and $j_{n+1}=1$.
So $i_{n+1}=1$ and $|\bfj|= m+n$, and all the vertices on the path $p$ are projective and injective, which is on the subquiver $Q(n)\times\{1\}$ of $Q(n+1)$.
By the inductive assumption, there are paths $q'$ and $p''$ in $Q(n)\times\{1\}$, such $q'pp''$ is a bound path from $\bfi'=(i_1',\ldots,i_{n+1}')$ to $\bfj'=(j_1',\ldots,j_{n+1}')$ of length $n$.
But $i'_t\ge 1$ and $j'_t\ge 1$ for $t=1,\ldots,n+1$, $q'pq''$ is formed by arrows of different type, thus $j'_{n}>1$ and $\bfj'  , \bfj'(n+1)$ are vertices in $Q(n+1)$.
So $\gkk{n+1}{\bfj'}:\bfj'  \to \bfj'(n+1)$ is  an arrow of $Q(n+1)$.
Let $p'=  \gkk{n+1}{\bfj'}q'$, then $p'pp''$ is a bound path of length $n+1$ in $Q(n+1)$.
This proves the Condition (i) of the admissibility.

Let $p =\gkk{t_s}{\bfi(t_1,\ldots,t_{s-1})}\cdots \gkk{t_1}{\bfi}$ be a bound path from non-injective vertex $\bfi$ to  non-projective vertex $\bfj$.
Then we have that $j_{n+1} > 1$ and $|\bfi| < m+n$.
We prove that $p$ is shiftable by using induction on $s$.

If $s=1$, then either $t_1=1$, then $i_{n+1}=j_{n+1}$, both $\bfi$ and $\bfj$ are non-projective and we take  $p'=p$, $p''=e_{\bfi}$, so $p=p'p''$ is shiftable, or $t_1\neq 1$ and $|\bfj|=|\bfi|<m+n$, thus both $\bfi$ and $\bfj$ are non-injective and we take  $p'=e_{\bfj}$, $p''=p$, so $p=p'p''$ is shiftable $p$.

Assume that $l \ge 1$ is an integer and any bound path of length $l$ from a non-injective vertex to a non-projective vertex is shiftable for $s=l$.

Now let $s=l+1$.
Assume that $p$ passes through both a projective and an injective vertex.
Since for any arrow $\alpha: \bfi' \to \bfj'$,  $\bfj'$ is injective if $\bfi'$ is so, and $\bfi'$ is projective if $\bfj'$ is so, thus $\bfi$ is projective and $\bfj$ is injective and we have $i_{n+1} = 1$ and $|\bfj| = m+n$.
This implies that there is one arrow of type $1$ one arrow of type $n+1$ in $p$ in the path, and we have $j_{n+1} =2$ and $|\bfi| =m+n-1 $.

If $t_1\neq 1$, $p$ does not start with an arrow of type $1$, then $\bfi(t_1)$ is not injective, then $p' =\gkk{t_s}{\bfi(t_1, \ldots, t_{s-1})} \cdots \gkk{t_2}{\bfi(t_1)} $ is a bound path from a non-injective vertex $\bfi(t_1)$ to a non-projective vertex $\bfj$ and thus it is shiftable by inductive assumption.
So $p$ is also shiftable by definition.

If $t_{l}\neq n+1$, that is, $p$ does not end with an arrow of type $n+1$, then $(t_l)\bfj$ is not projective, then $p' =\gkk{t_{l}}{\bfi(t_1,\ldots,t_{l-1})}\cdots \gkk{t_1}{\bfi} $ is a bound path from a non-injective vertex $\bfi$ to a non-projective vertex $(t_{l})\bfj$ and thus it is shiftable by inductive assumption.
So $p$ is also shiftable by definition.

Now assume that $t_1= 1$, $t_{l+1} = n+1$.

If $i_{t_{r}-1}=1$ for $1<r\le l+1$ and $j_{t_r}=1$ for $1\le r \le l$.
Then for $r=1,\ldots, l+1$, $\bfi(t_1,\ldots,t_r) \in Q(n+1)$ if and only if $t_v=v$ for $v=1,\ldots,r$ and similarly, for $r=1, \ldots,l$, $(t_r,\ldots,t_{l+1})\bfj \in Q(n+1)$ if and only if $t_{l+1-v}=n+1-v$ for $v=1, \ldots,l+1-r$.
This implies $l=n$ and $j_{n+1} = i_{n+1}+1= m$, since $\bfj$ is injective and $|\bfj|=m+n$.
So $ i_{n+1}=m-1\ge 2$, this contradicts the fact that $\bfi$ is projective.

So there is $1< r \le l+1$ such that $i_{t_r-1}>1$, or there is $1\le r \le l$, such that $j_{t_r}>1$.
In the first case we have that $$p = d \gkk{t_{l+1}}{\bfi(t_1,\ldots,t_{l})}\cdots \gkk{t_{r+1}}{\bfi(t_1,\ldots,t_{r})} \gkk{t_{r-1}}{\bfi(t_1,\ldots,t_{r-2},t_r)} \cdots \gkk{t_1}{\bfi(t_r)}\gkk{t_r}{\bfi}$$ by Lemma ~\ref{commpath}, and $p$ is a multiple of a path which does not start with an arrow of type $1$.
Similarly, in the second we have that $$p = d \gkk{t_r}{\bfi(t_1,\ldots, t_{r-1}, t_{r+1}, \ldots, t_{l+1})} \gkk{t_{l+1}}{\bfi(t_1, \ldots, t_{r-1}, t_{r+1}, \ldots, t_{l})} \cdots \gkk{t_{r+1}}{\bfi(t_1,\ldots,t_{r-1})} \gkk{t_{r-1}}{\bfi(t_1, \ldots, t_{r-2})} \cdots \gkk{t_1}{\bfi(t_r)}.$$
Thus $p$ is linearly dependent to shiftable paths by above argument and (ii) of the admissible condition hold.

It follows from Lemma~\ref{pathlengthl} that for any pair $\bfi, \bfj$ of vertices of $Q(n+1)$, all the bound paths from $i$ to $j$ are linearly dependent if there is one, it is easy to see that (iii) of the admissible condition holds.

This proves that a pyramid $n$-cubic algebra is $(n-1)$-translation algebra with admissible $(n-1)$-translation quiver, by induction.
By Proposition 4.2 of \cite{gu14}, there is $\bbfg $ such that  $\tL^{\bbfg}(n+1)$ is the trivial extension of $\LL(n+1)$ and by Theorem ~\ref{ncAkszul},  $\LL(n+1)$ is extendible.
\end{proof}

\section{$n$-almost Split Sequences and Absolutely $n$-complete Algebras}
We omit $\bfg, \bfg' $ and $\bbfg$ whenever it is possible in this section.
Let $\GG(n+1)={\LL(n+1)^!}^{op}$ be the Koszul dual of $\LL(n+1)$.
They have the same quiver with quadratic dual relations.
The Koszul complexes gives a correspondence between the radical layers of the projective cover of a simple and terms in the projective resolution of the corresponding simple of these two algebras \cite{bbk}.
The radical layers of the projective cover and the terms of projective resolution of a simple $\LL(n)$-module are described by the $n$-cubic cells $H^{\bfi}$ and $n$-hammocks $H^{\bfi}_C$ (see Lemma ~\ref{projrefine} and \eqref{bhprtmlnpo}).

Now we consider the $n$-almost split sequence in the category of the finite generated projective modules.
We first study the $n$-cubic cells and $n$-hammocks, we have the following lemma.

\begin{lemma}\label{cuboidcube}
Let $Q(n)$ be a pyramid shaped $n$-cubic quiver and let $\bfi$ be a vertex of $Q(n)$.

If $H^{\bfi}$ is complete, then $H_C^{\bbvi{\bfe}}$ is not complete.

If $H_C^{\bfi}$ is complete, then $H^{\bbvi{\bfb(\bfi)}}$ is not complete.
\end{lemma}
\begin{proof}
For $\bfi= (i_1,\ldots, i_n) \in G(n)_0$, then $\bbvi{\bfe} =(i_1, \ldots,i_{n-1}, i_n+1 )$.
Thus $\bfb(\bbvi{\bfe}) = (m+n-1-\sum_{t=1}^n i_t,i_2-1,\ldots,i_{n-1}-1,i_n)$
and $$\bar{\bfv}^{\bbvi{\bfe}}(\bfb(\bbvi{\bfe})) =(i_1+m+n-1-\sum_{t=1}^n i_t-i_2+1, 2i_2-i_3,\ldots,2i_{n-1}-i_{n}-1,2i_n+1 ),$$
and we have $$|\bar{\bfv}^{\bbvi{\bfe}}(\bfb({\bbvi{\bfe}}))| = m+n-1+1= m+n.$$
This implies that $\bar{\bfv}^{\bbvi{\bfe}}(\bfb({\bbvi{\bfe}})) \not \in Q(n)_0$ and $H_C^{\bbvi{\bfe}}$ is not complete.

The second statement is proven similarly.
\end{proof}

We  have the following lemma.

\begin{lemma}\label{prohammockL}
The following are equivalent for $\LL(n+1)$.
\begin{itemize}
\item[(i)] $\LL(n+1)e_{\bfi}$ is projective injective.
\item[(ii)] The Loewy length of $\LL(n+1)e_{\bfi}$ is $n+2$.
\item[(iii)] $H^{\bfi}$ is complete.
\end{itemize}
\end{lemma}
\begin{proof}
$(i) \Leftrightarrow (ii)$. Since $\LL(n+1)$ is $(n+1)$-translation algebra, an projective module is  projective injective module if and only if its Lowey length is $n+2$.

$(ii) \Leftrightarrow (iii)$. The Loewy length of $\LL(n+1)e_{\bfi}$ is $n+2$ if  and only if there is a bound path of length $n+1$ starting from $\bfi$ in  $(Q(n+1),\rho(n+1))$.
The only bound path of length $n+1$ in starting from $\bfi$ in  $(Q(n+1) ,\rho(n+1))$ must end at $\bbvi{\bfe}$.
So the Loewy length of $\LL(n+1)e_{\bfi}$ is $n+2$  if  and only if $\bbvi{\bfe}$ in  $(Q(n+1), \rho(n+1))$, that is, if  and only if  $H^{\bfi}$ is complete.
\end{proof}

It follows from Theorem 6.2 of \cite{gu14}, that ${\tL(n)^!}^{op}$ and ${\olL(n)^!}^{op}$ are self-injective of Loewy length $m$, so they are $(m-2)$-translation algebras.
Via the Koszul complexes, $(n+1)$-hammocks can be used to describe the radical layers of the projective covers and $(n+1)$-cubic cells can be used to describe the projective resolutions of the simple ${\olL(n)^!}^{op}$-modules.
Note that  $\GG(n+1)={\LL(n+1)^!}^{op}$ have the same quiver $Q(n+1)$ as $\LL(n+1)$, the arrows are the opposites of the dual basis of the original arrows of $k Q(n+1)_1$, which we denote by the same notations.
The relations can be chosen as a basis of the orthogonal space of the subspace spanned by $\rho^{\bbfg}(n+1)$, say  \eqqc{perprelation}{ \arr{lll}{\rho^{\bbfg, \perp}(n+1) &=& \{ d_{s,t,\bfi}^{-1} \gkk{t}{\bfi({s})} \gkk{s}{\bfi} + \gkk{s}{\bfi({t})} \gkk{t}{\bfi} \mid  \bfi,\bfi(t), \bfi({s}), \bfi(t)({s})\in Q(n+1)_0, \\ && \qquad \qquad  1\le t < s \le n+1\} \\ &&  \quad \cup \{ \gkk{s+1}{\bfi{s}}\gkk{s}{\bfi}\mid s=1 \mbox{ or } i_{s-1}>0, i_{s}=1\}.}}
As a truncation of ${\olG(n)^!}^{op}$,  we also have the following version of Lemma ~\ref{prohammockL} for $\GG(n+1)$.
\begin{lemma}\label{prohammockG}
$\GG(n+1)$ is an $(m-2)$-translation algebra, and the following are equivalent for $\GG(n+1)$.
\begin{itemize}
\item[(i)] $\GG(n+1)e_{\bfi}$ is projective injective.
\item[(ii)] The Loewy length of $\GG(n+1)e_{\bfi}$ is $m$.
\item[(iii)] $H_C^{\bfi}$ is complete.
\end{itemize}
\end{lemma}
\begin{proof}
It follows from Lemma ~\ref{boundsubq} that $(Q(n+1), \rho^{\bbfg, \perp}(n+1))$ is a full bound subquiver of $(\olQ(n), \olrho^{\bbfg, \perp}(n))$, the bound quiver of $\olG(n)= {\olL(n)^!}^{op}$, and it is an $(m-2)$-translation quiver.
Thus $\GG(n+1) \simeq \olG(n)/I'(t)$ where $i'(t)$ is the ideal of $\olG(n)$ generated by the set $\{e_{\bfi} \mid \bfi \in \olQ(n)_0\setminus Q(n+1)(t)_0 \}$.
Since it is the quadratic dual of $(n+1,m-1)$-Koszul algebra, it is an $(m-1,n+1)$-Koszul algebra, hence an $(m-2)$-translation algebra.

Thus $\GG(n+1)e_{\bfi}$ is projective injective if and only if its Loewy length is $m$, and (i) is equivalent to (ii).

Note that all the bound path from $\bfi$ are inside $H_C^{\bfi}$, and an bound path of length $m-1$ from $\bfi$ ends at $\bbvi{\bfb(\bfi)}$, thus the Loewy length of $\GG(n+1)e_{\bfi}$ is $m$ if and only if $\bbvi{\bfb(\bfi)}$ is in $H_C^{\bfi}$, that is $H_C^{\bfi}$ is complete.
This proves that (ii) is  is equivalent to (iii).
\end{proof}

Note that the $n$-translation of $\LL(n+1)$ is defined by $\tau_{n}^{-1} \bfi = \bfi+\bfe$ and the $(m-2)$-translation of $\GG(n+1)$ is defined by $\tau_{[m-2]}^{-1}\bfi =\bbvi{\bfb(\bfi)}$.
Then we have the following result:
\begin{theorem}\label{nalmost}
$\GG(n+1)$ is an $(m-2)$-translation algebra.

$\add\GG(n+1)$ has $n$-almost split sequence.

$\add\LL(n+1)$ has $(m-3)$-almost split sequence.
\end{theorem}
\begin{proof}
The first assertion follows from Lemma ~\ref{prohammockG}.

Since $\GG(n+1)$ is the quadratic dual of $\LL(n+1)$ and $\LL(n+1)$ is an $n$-translation algebra, $\GG(n+1)$ is a partial Artin-Schelter $n$-regular
algebra, by Theorem 6.1 of \cite{gu14}.
Similarly, $\LL(n+1)$ is a partial Artin-Schelter $(m-2)$-regular algebra.
The last two assertions follow from Theorem 7.2 of \cite{gu14}.
\end{proof}

The following proposition gives the relationship between the quadratic dual $\GG(n+1)$ of $\LL(n+1)$ and Iyama's  absolutely $(n + 1 )$-complete algebra.
\begin{proposition}\label{compabsnalg}$\GG^{\bbfg}(n+1) = T_m\sk{n+1}(k)$ for some $\bbfg$.
\end{proposition}
\begin{proof}
Starting with the  quiver \eqref{quiver0}, 
we have that $\GG(1)=kQ(1)^{op}$.
Consider the quiver $Q\sk{n+1}$ defined in Definition 6.11 of \cite{i4}.
For $x \in Q(1)_0$, we  have   $l_x+x=m+1$ for $x\in Q(0)_0$, since $\tau^{m+1-x} I(x)=0$ for the injective $kQ$-module $I(x)$ corresponding to the vertex $x$.
One sees easily, $\iota: \bfi \to \bfi - \bfe_2- \cdots - \bfe_{n+1}$ defines a bijection from $Q(n+1)_0$ to $Q\sk{n+1}_0$,  and the arrows are reversed under this bijection.
So $Q(n+1)$ is the opposite quiver of $Q\sk{n+1}$.

Choosing $\bbfg$ in \eqref{perprelation} such that $d_{s,t,\bfi} =-1$ for all $s,t, \bfi $ with  $ \bfi,\bfi(t), \bfi({s}), \bfi(t)({s})\in Q(n+1)_0$, that is for each square with vertex set $\bfi,\bfi(t), \bfi({s}), \bfi(t)({s})$, we get a commutative relation for each square in $Q(n+1)$ and a zero relation for each half square in $Q(n+1)$.
\eqref{perprelation} is the set of the opposites of relations given in  Theorem 6.12 of \cite{i4} for the above quiver $Q\sk{n+1}$.
Thus $(Q(n+1), \rho^{\bbfg,\perp})$ is exactly the quiver and relations defining $T_m\sk{n+1}(k)$ and $\GG^{\bbfg}(n+1) \simeq T_m\sk{n+1}(k)$.
\end{proof}

By the above proposition, we have the following realization of Iyama's cone construction of absolutely $(n + 1 )$-complete algebras from an absolutely $n$-complete algebra.

\begin{theorem}\label{abscone}
There exist $\bfg $ such that $\LL^{\bfg}(n) = {T_m\sk{n}}^{!op}(k)$ is the pyramid $n$-cubic algebra defined in Proposition ~\ref{abs}, and there is an $(n+1)$-cuboid completion $\LL^{\bbfg}(n+1)$ of $\LL^{\bfg}(n) $ such that $ {T_m\sk{n+1}}(k) \simeq {\LL^{\bbfg}}^{!op}(n+1)$.
\end{theorem}

{}
\end{document}